\documentclass[a4paper,12pt]{article}
\usepackage{amsmath,amsthm,amssymb,amsfonts}
\usepackage{bbm,bm}
\usepackage{latexsym}
\usepackage{mathrsfs}
\usepackage{threeparttable}
\usepackage{tabularx}
\usepackage{booktabs}
\usepackage{graphicx}
\usepackage{epstopdf}
\usepackage{color}
\usepackage{indentfirst}
\usepackage{mathrsfs}
\usepackage{geometry}
\usepackage{caption}
\usepackage{lineno}
\usepackage{cleveref}
\usepackage{abstract}
\usepackage{enumerate,mdwlist}
\usepackage{cases}

\usepackage[numbers,sort&compress]{natbib}

\theoremstyle{plain}
\newtheorem{The}{Theorem}[section]    % 分章定义
\newtheorem{Lem}[The]{Lemma}

\newtheorem{Rem}[The]{Remark}
\newtheorem{Con}[The]{Conjecture}
\theoremstyle{definition}

\numberwithin{equation}{section} %公式编号带章节

\makeatletter

\newcommand{\Rmnum}[1]{\expandafter\@slowromancap\romannumeral #1@}
\makeatother
\title{The star-structure connectivity and star-substructure connectivity of hypercubes and folded hypercubes }

\author{Lina Ba and Heping Zhang\thanks{Corresponding author.}}
\date{{\small School of Mathematics and Statistics, Lanzhou University,
 Lanzhou, Gansu 730000, P.R. China}\\
{\small E-mails:\ baln19@lzu.edu.cn, zhanghp@lzu.edu.cn}}

\begin{document}

%-----------------Title------------------
\maketitle
%----------------Abstract------------------
\begin{abstract}
As a generalization of vertex connectivity,  for connected graphs $G$ and $T$, the $T$-structure connectivity $\kappa(G, T)$ (resp. $T$-substructure connectivity $\kappa^{s}(G, T)$) of $G$ is the minimum cardinality of a set of subgraphs $F$  of $G$ that each  is isomorphic to $T$ (resp.  to a connected subgraph of $T$) so that  $G-F$ is disconnected. For $n$-dimensional hypercube $Q_{n}$, Lin et al. \cite {4} showed $\kappa(Q_{n},K_{1,1})=\kappa^{s}(Q_{n},K_{1,1})=n-1$ and $\kappa(Q_{n},K_{1,r})=\kappa^{s}(Q_{n},K_{1,r})=\lceil\frac{n}{2}\rceil$ for $2\leq r\leq 3$ and $n\geq 3$.  Sabir et al. \cite {9} obtained that $\kappa(Q_{n},K_{1,4})=\kappa^{s}(Q_{n},K_{1,4})=
    \lceil\frac{n}{2}\rceil$ for $n\geq 6$, and for $n$-dimensional folded hypercube $FQ_{n}$, $\kappa(FQ_{n},K_{1,1})=\kappa^{s}(FQ_{n},K_{1,1})=n$, $\kappa(FQ_{n},K_{1,r})=\kappa^{s}(FQ_{n},K_{1,r})=
    \lceil\frac{n+1}{2}\rceil$ with $2\leq r\leq 3$ and $n\geq 7$. They proposed an open problem of determining $K_{1,r}$-structure connectivity of  $Q_n$ and $FQ_n$ for general $r$. In this paper,
  we  obtain that  for each integer $r\geq 2$, $\kappa(Q_{n};K_{1,r})$$=\kappa^{s}(Q_{n};K_{1,r})$
 $=\lceil\frac{n}{2}\rceil$ and $\kappa(FQ_{n};K_{1,r})=\kappa^{s}(FQ_{n};K_{1,r})= \lceil\frac{n+1}{2}\rceil$ for all integers $n$ larger than  $r$ in quare scale. For $4\leq r\leq 6$, we separately confirm the above result holds for $Q_n$ in the remaining cases.
     \vskip 0.1in

    \noindent {\bf Keywords:} \ Structure connectivity;  Substructure connectivity;  Star graph; Hypercube; Folded hypercube.

    \vskip 0.1 in
\medskip
\end{abstract}
%----------------Abstract end------------------
\section{Introduction}
It is well known that the topology of an interconnection network is often modeled by a connected graph. Let $G$ be a graph with vertex set $V(G)$ and edge set $E(G)$, where each vertex  represents a processor or node and every edge a communication link. For a subgraph $H$ of $G$, we use $G-H$ to denote the subgraph  $G-V(H)$. For a set $F=\{T_{1}, T_{2}, \ldots , T_{m}\}$ of  subgraphs of $G$, let $G-F= G-V(T_{1})-V(T_{2})-\ldots-V(T_{m})$. A good interconnection network should have some good performances, such as uniformity, symmetry, high fault tolerance, expansibility and small fixed vertex degree. One of the important parameters of high fault tolerance is connectivity.
A vertex-cut of a graph $G$ is a set $S\subseteq V(G)$ such that $G-S$ has more than one component. The connectivity $\kappa(G)$ of $G$ is defined as the minimum cardinality of a vertex-cut $S$ such that $G-S$ is disconnected or has only one vertex. In 1994, Fabrega et al. \cite {3} proposed $g$-extra connectivity,   providing more accurate measures for fault tolerance of  large-scale parallel processing systems.  For a connected non-complete graph $G$ and a non-negative integer $g$, a vertex cut $S$ of $G$ is an $g$-extra cut if $G-S$ is disconnected and every component of $G-S$ has more than $g$ vertices. The $g$-extra connectivity $\kappa_{g}(G)$ of $G$ is defined as the minimum cardinality of $g$-extra cut of $G$.

 In reality of network reliability and fault-tolerance, the neighbors of a faulty node might be more vulnerable. For networks and subnetworks  made into chips, when any node on the chip becomes faulty, the whole chip can be considered faulty.  To study the fault-tolerance of some structures of the network, Lin et al. \cite {4} introduced the concepts of structure connectivity and substructure connectivity of networks. Let $T$ be a connected subgraph of graph $G$. Let  $F$ be a set of subgraphs of $G$ such that every member in $F$ is isomorphic to $T$. Then $F$ is called a $T$-{\em structure-cut} of $G$ if the deletion of all members of $F$ disconnects $G$, i.e. $G-F$ is disconnected. The $T$-{\em structure connectivity} $\kappa(G,T)$ of $G$ is defined as the minimum cardinality of a $T$-structure-cut of $G$. Similarly,  a set $F'$ of subgraphs of $G$ which each is isomorphic to a connected subgraph of $T$  is called a $T$-{\em substructure-cut} if $G-F'$ is disconnected. The $T$-{\em substructure connectivity} $\kappa^{s}(G,T)$ of $G$ is defined as the minimum cardinality of a $T$-substructure-cut of $G$. Figure 1 shows an example of $T$-structure-cut and $T$-substructure-cut where $T$ is 3-cycle $C_{3}$. By definition, $\kappa^{s}(G,T)\leq \kappa(G,T)$. Note that $K_{1}$-structure connectivity  reduces to the classical vertex connectivity.

\begin{figure}[!htbp]
\begin{center}
\includegraphics[totalheight=5cm]{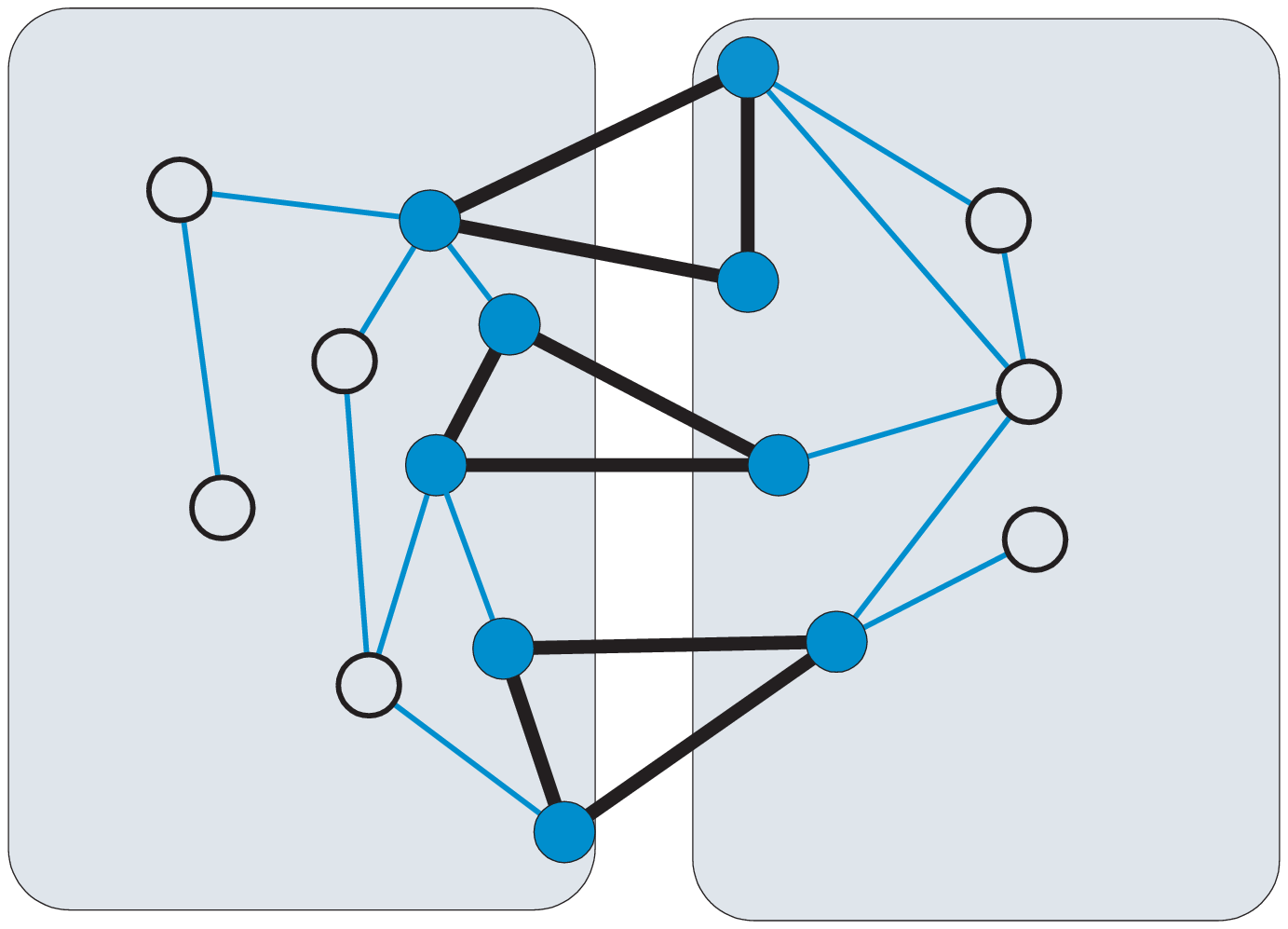}
\includegraphics[totalheight=5cm]{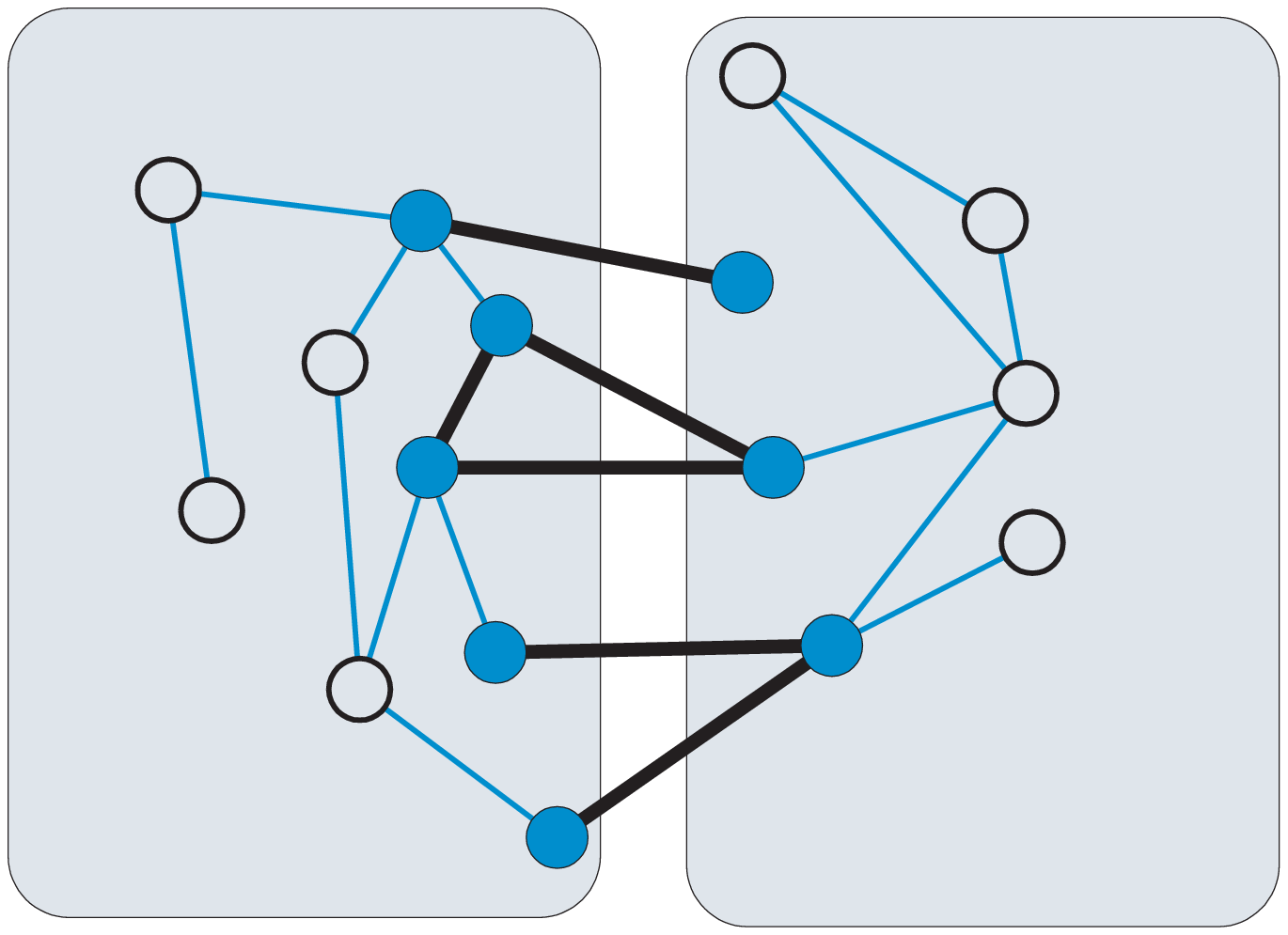}
\caption{\label{JGG }\small{ $C_{3}$-structure cut and  $C_{3}$-substructure cut.}}
\end{center}
\end{figure}

In the study of $T$-structure connectivity, much of the work has been focused on certain special structures of some given networks. Let $P_k$ denote a path with $k$ vertices, $C_k$ a cycle with $k$ vertices, and $K_{1,r}$ a star with $r\geq 1$ leaves. For the bubble-sort star graph $BS_{n}$, Zhang et al. \cite {13} obtained  $\kappa(BS_{n},T)$ and $\kappa^{s}(BS_{n},T)$ for $T\in\{P_{k},C_{2k}\}$. For $k$-ary $n$-cube network $Q_{n}^{k}$, Lv et al. \cite {7} showed $\kappa(Q_{n}^{k},K_{1,r})$ and $\kappa^{s}(Q_{n}^{k},K_{1,r})$ with $1\leq r\leq3$;  further, Lu et al. \cite {17} showed $\kappa(Q_{n}^{k},T)$ and $\kappa^{s}(Q_{n}^{k},T)$ for $T\in\{P_{k},C_{k}\}$ where $3\leq k\leq2n$;
For $n$-dimensional twisted hypercube $TQ_{n}$, Li et al. \cite {6} obtained $\kappa(TQ_{n},T)$ and $\kappa^{s}(TQ_{n},T)$ for $T\in\{K_{1,3},K_{1,4}, P_{k}\}$ where $1\leq k\leq n$.

For $n$-dimensional  hypercube  $Q_{n}$, Lin et al. \cite {4} showed
%\begin{equation*}
\begin{eqnarray}\label{1}
\kappa(Q_{n},K_{1,r})=\kappa^{s}(Q_{n},K_{1,r})=
\left\{ \begin{aligned}
&n-1, &\mbox{if}~~r=1,~n\geq3,\\
&\lceil\frac{n}{2}\rceil,&\mbox{if}~~2\leq r\leq3,~n\geq3.
\end{aligned}\right.
\end{eqnarray}
%\end{equation*}

Moreover,
Sabir et al. \cite {9} established

\begin{eqnarray}\label{2}
 \kappa(Q_{n};K_{1,4})=\kappa^{s}(Q_{n};K_{1,4})=\lceil\frac{n}{2}\rceil, \mbox{for } n\geq 6;
\end{eqnarray}

and for $n$-dimensional folded hypercubes $FQ_{n}$, they also determined for $n\geq7$,
\begin{equation*}
\label{}
\kappa(FQ_{n},K_{1,r})=\kappa^{s}(FQ_{n},K_{1,r})=
\left\{ \begin{aligned}
&n, &\mbox{if}~~r=1,\\
&\lceil\frac{n+1}{2}\rceil, &\mbox{if}~~r=2,3.
\end{aligned}\right.
\end{equation*}

From the above results we can see that for $Q_n$ and $FQ_n$ the structure connectivity of only small stars $K_{1,r}$ $(1\leq r\leq 4)$ have been already determined. So Sabir et al. \cite {9} pointed out that determining  the $K_{1,r}$-structure connectivity and $K_{1,r}$-substructure connectivity of $Q_n$ and $FQ_n$  with general $r$ remain open. In this paper, we treat general star-structure connectivity for $n$-dimensional  hypercube  $Q_{n}$ and folded hypercubes $FQ_n$ and obtain the following results: for each integer $r\geq 2$, $\kappa(Q_{n};K_{1,r})$$=\kappa^{s}(Q_{n};K_{1,r})$
 $=\lceil\frac{n}{2}\rceil$ and $\kappa(FQ_{n};K_{1,r})=\kappa^{s}(FQ_{n};K_{1,r})= \lceil\frac{n+1}{2}\rceil$ for all integers $n$ larger than  $r$ in quare scale. To describe clearly the extent of $n$ exceeding $r$ we introduce two functions $f(r)$ and $g(r)$. For details, see  Theorems \ref{hypercube} and \ref{FQ}. Such results partly solve the open problem. For $Q_n$, from the above-mentioned results (\ref{1}), (\ref{2}) and Theorems \ref{hypercube} we find that the $K_{1,r}$-structure  and substructure connectivity of $Q_n$ for $4\leq r\leq 6$ and $n=r$ and $r+1$ have not been determined yet. So in section 4, we separately discuss such low dimensional cases and get the same result. That is, for $4\leq r\leq6$ and $n\geq r$ we have that, $\kappa(Q_{n};K_{1,r})=\kappa^{s}(Q_{n};K_{1,r})=\lceil\frac{n}{2}\rceil$.
\section{Preliminaries}
We only consider finite and simple graphs $G$.  Two vertices $u$ and $v$ of $G$ are adjacent  if they are the end-vertices of an edge. A neighbor of a vertex $x$ of $G$ means a vertex of $G$ adjacent to $x$. The neighborhood of a vertex $x$ in $G$ is the set of neighbors of $x$, denoted by $N_{G}(x)=\{y|xy\in E(G)\}$. The neighborhood of a vertex set $A$ in $G$ is denoted by $N_{G}(A)=\cup_{x\in A}N_{G}(x)-A$.
A path $P_k=v_{1}v_{2}\ldots v_{k}$ of length $k-1$ is a sequence of $k$ distinct vertices such that $v_{i-1}v_{i}\in E(G)$ for every $2\leq i\leq k$. If the end-vertices of a path $P$ of length $k\geq3$ are identified, then it becomes a cycle of length $k$, denoted by $C_k$.

An $n$-dimensional hypercube $Q_{n}$ is a simple graph on the all binary strings of length $n$, such two strings $u_{1}u_{2}\cdots u_{n}$ and $u'_{1}u'_{2}\cdots u'_{n}$,  $u_{i},u_i'\in\{0,1\}$ for $1\leq i\leq n$, are adjacent if and only if they differ in exactly one position \cite {18}, that is, $\sum_{i=1}^n|u_i-u_i'|=1$. For any vertex $u=u_{1}u_{2}u_{3}\ldots u_{n}$ in $Q_{n}$, we set $u^{i}=u_{1}^{i}u_{2}^{i}u_{3}^{i}\ldots u_{n}^{i}$ is the neighbor of $u$ in dimension $i$ of $Q_{n}$ where $u_{j}^{i}=u_{j}$ for $j\neq i$ and $u_{i}^{i}=1-u_{i}$. In general, for $A\subseteq \{1,2,\ldots, n\}$, let $u^A$ be the vertex of $Q_n$ so that $(u^A)_i=\overline{u_i}=1-u_i$ if and only if $i\in A$. Obviously, for $A,B\subseteq \{1,2,\ldots, n\}$, $u^{A}=u^{B}$ if and only if $A=B$. So $u^{i_{1},i_{2}}$ is the neighbor of $u^{i_{1}}$ in dimension $i_{2}$ and  $u^{i_{1},i_{2},i_{3}}$ is the neighbor of $u^{i_{1},i_{2}}$ in dimension $i_{3}$. We make a convention:  the  elements in $\{1,2,\ldots,n\}$ are taken arithmetic operations on module $n$.  It is known that $Q_{n}$ is a bipartite and $n$-regular graph.
\begin{figure}[!htbp]
\begin{center}
\includegraphics[totalheight=5cm]{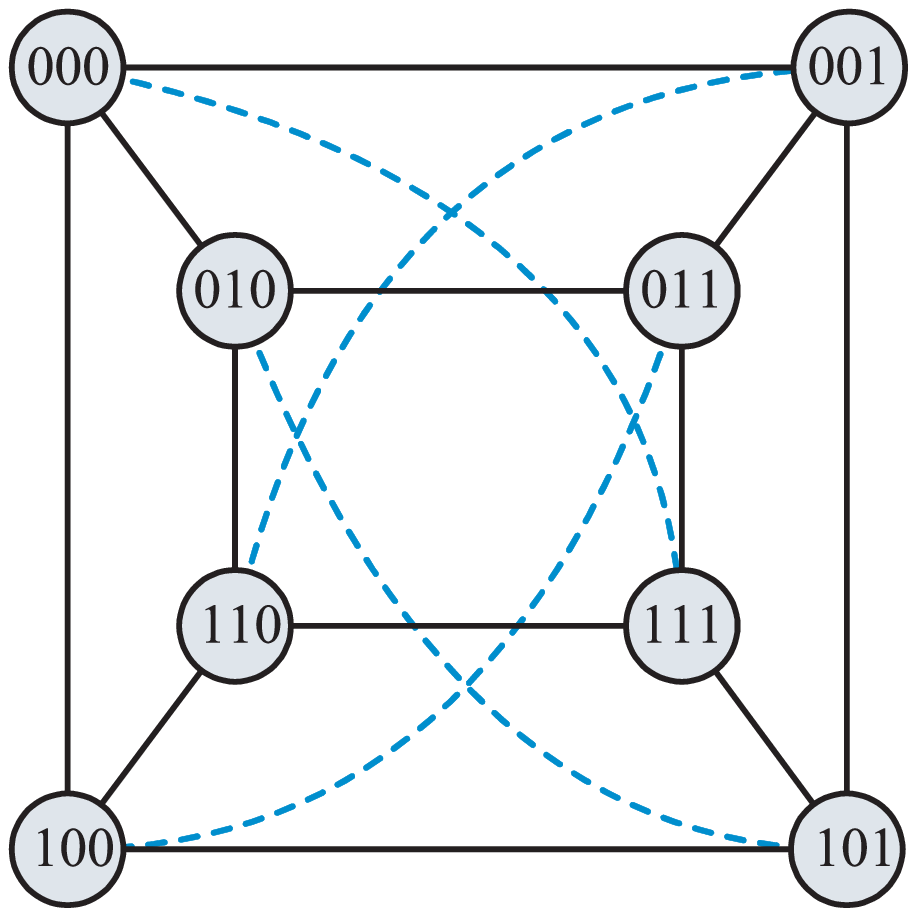}
~~~~~~~~~~~~~~~~~~~~~~\includegraphics[totalheight=5cm]{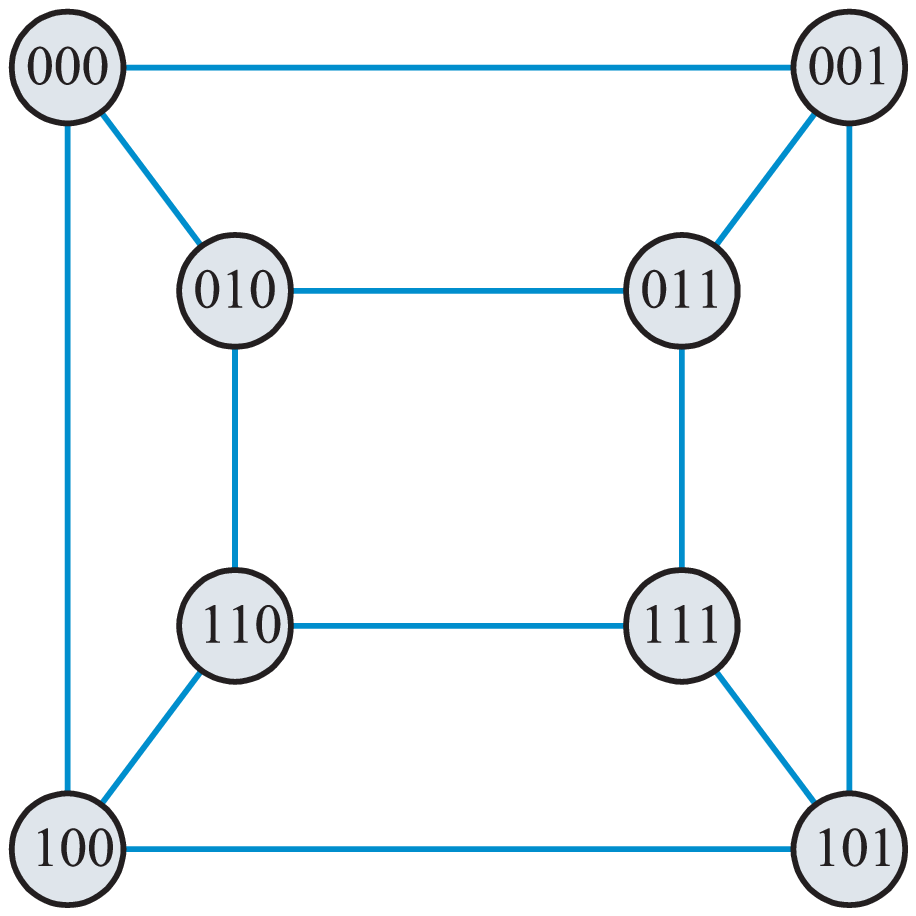}
\caption{\label{FQ3 }\small{ $FQ_{3}$ and  $Q_{3}$.}}
\end{center}
\end{figure}

As one of the popular variants of the hypercube, the $n$-dimensional folded hypercube $FQ_{n}$ proposed by El-Amawy and Latifi \cite {1}  is a graph obtained from hypercube $Q_{n}$ by adding $2^{n-1}$ edges, each of them being between vertices $u=u_{1}u_{2}u_{3}\ldots u_{n}$ and $\overline{u}=\overline{u_{1}}\overline{u_{2}}\overline{u_{3}}\ldots \overline{u_{n}}$, where $\overline{u_{i}}=1-{u_{i}}$.  $FQ_{n}$ is a highly symmetric graph as a underlying topology of several parallel systems, such as ATM Switches \cite {8}, PM21 networks \cite {5} and 3D-FolH-NOC network \cite {2}. For example, the $FQ_{3}$ and $Q_{3}$ are illustrated in Figure 2.

\section{The star-structure connectivity of hypercubes }
To determine the star-structure connectivity and star-substructure connectivity of $n$-hypercubes, we first list some preliminary results.
\begin{Lem}{\rm\cite{10}}
Any two vertices in $Q_{n}(n\geq 3)$ have exactly two common neighbors, if they have any.
\end{Lem}
The following two lemmas in the case $3\leq r\leq n$ are taken from Lemmas 2.4 and 2.5 in reference {\rm\cite {9}} respectively. We find that they also hold for $r=2$ by Lemma 3.1, since $Q_{n}$ is triangle-free.
\begin{Lem}
Let $K_{1,r}$ be a star in $Q_{n}$ with $2\leq r\leq n$. If $u$ is a vertex in $Q_{n}-K_{1,r}$, then
$|N_{Q_{n}}(u)\cap V(K_{1,r})|\leq 2$, and equality holds if and only if $u$ is adjacent two leaves of $K_{1,r}$.
\end{Lem}

\begin{Lem}\label{adjacent}
Let $K_{1,r}$ be a star in $Q_{n}$ with $2\leq r\leq n$. If $u$ and $v$ are two adjacent vertices of $Q_{n}-K_{1,r}$, then $|N_{Q_{n}}({u, v})\cap V(K_{1,r})|\leq2$.
\end{Lem}
Now we extend two adjacent vertices $u$ and $v$ in Lemma 3.3 to a connected subgraph $C$ in $Q_{n}-K_{1,r}$ with $|V(C)|\geq2$ as follows.
\begin{Lem}\label{connection}
Let $K_{1,r}$ be a star in $Q_{n}$ with $2\leq r\leq n$. If $C$ is a connected subgraph in $Q_{n}-K_{1,r}$ with $k=|V(C)|\geq2$, then
$|N_{Q_{n}}(C)\cap V(K_{1,r})|\leq 2(k-1)$, and  equality holds only if $C$ is a star in $Q_{n}$.
\begin{proof}Let $V(K_{1,r})=\{x,x_{1},x_{2},\ldots, x_{r}\}$ and $E(K_{1,r})=\{xx_{i}|1\leq i\leq r\}$. Then $x$ is the center of $K_{1,r}$. Let $V(C)=\{u_{1},u_{2},\ldots,u_{k}\}$.

First, we prove that $|N_{Q_{n}}(C)\cap V(K_{1,r})|\leq 2(k-1)$.
Suppose to the contrary that $|N_{Q_{n}}(C)\cap V(K_{1,r})|\geq 2(k-1)+1=2k-1$.
By Lemma $3.2$, each vertex $u_i$ in $C$ has at most 2 neighbors in $K_{1,r}$, and if $|N_{Q_{n}}(u_{i})\cap V(K_{1,r})|=2$, then $u_{i}$ is adjacent to two leaves in $K_{1,r}$, so $2k \geq|N_{Q_{n}}(C)\cap V(K_{1,r})|\geq 2k-1$. It means that there exists at least $k-1$ vertices in $C$ which each has two neighbors in $K_{1,r}$, and such neighbors are pairwise distinct. Without loss of generality, we assume
$\{u_{i}x_{2i-3},u_{i}x_{2i-2}\}\subset E(Q_{n})$ for $2\leq i\leq k$. Since $C$ is connected, $u_{1}$ is adjacent to $u_{i}$ for some $2\leq i\leq k$. If $u_{1}x_{j}\in E(Q_{n})$ with  $2k-1\leq j\leq r$, then there exists an odd cycle $u_{1}x_{j}xx_{2i-3}u_{i}u_{1}$, a contradiction. Otherwise, $u_{1}x\in E(Q_{n})$. Then
 $N_{Q_{n}}(u_{i})\cap N_{Q_{n}}(x)=\{x_{2i-3},x_{2i-2},u_{1}\}$, contradicting  Lemma 3.1. Hence
 $|N_{Q_{n}}(C)\cap V(K_{1,r})|\leq 2(k-1)$.

 Next we show that if $|N_{Q_{n}}(C)\cap V(K_{1,r})|= 2(k-1)$, then $C$ is a star in $Q_{n}$. Suppose to the contrary that $C$ is not a star in  $Q_{n}$.
 Then we have that there exists a 4-vertex path $P_{4}$ in $C$ by taking a longest path of $C$, so  $4\leq k$
 and $6\leq|N_{Q_{n}}(P_{4})\cap V(K_{1,r})|\leq8$ by Lemma 3.2. However, by Lemma \ref{adjacent},
 any two consecutive vertices in $P_{4}$ together have at most two  neighbors in $V(K_{1,r})$, which implies that $P_{4}$ has at most four  neighbors in $V(K_{1,r})$, a contradiction.
\end{proof}
\end{Lem}

Yang et al. came to the following two results in  the $g$-extra connectivity of $Q_{n}$.
\begin{Lem}{\rm\cite {12}}
Let $C$ be a subgraph of $Q_{n}$ with $|V(C)|=g+1$
 for $n\geq 4$. Then $|N_{Q_{n}}(C)|\geq (g+1)n-2g-\tbinom{g}{2}$.
\end{Lem}

\begin{Lem}{\rm\cite {12}} For $n\geq 4$,

\[\kappa_{g}(Q_{n})=
            \begin{cases}
            (g+1)n-2g-\tbinom{g}{2},  &\text{if $0\leq g\leq n-4$}, \\
            \frac{n(n-1)}{2},   &\text{if $n-3\leq g\leq n$}. \\
            \end{cases}\]
\end {Lem}

\begin{Lem}\label{upper}
For $n\geq r \geq 2$ and $n\geq3$, $\kappa(Q_{n};K_{1,r})\leq \lceil\frac{n}{2}\rceil$ and
$\kappa^{s}(Q_{n};K_{1,r})\leq \lceil\frac{n}{2}\rceil$.
\begin{proof}Since $\kappa^{s}(Q_{n};K_{1,r})\leq \kappa(Q_{n};K_{1,r})$, we only prove
$\kappa(Q_{n};K_{1,r})\leq \lceil\frac{n}{2}\rceil$.
 Let $u=000\cdots 0$ be a vertex in $Q_{n}$. Then $N_{Q_{n}}(u)=\{u^{i}|1\leq i\leq n\}$.

If $n\geq 3$ is odd,
let $S_{i}=\{u^{2i-1},u^{2i},u^{2i-1,2i}\}\cup \{u^{2i-1,2i,2i+j}|1\leq j\leq r-2\}$  for $1\leq i\leq\frac{n-1}{2}$, and let  $S_{\frac{n+1}{2}}=\{u^{n},u^{n,1}\}\cup\{u^{n,1,j}|2\leq j\leq r\}$ for $n>r$ and $S_{\frac{n+1}{2}}=\{u^{n},u^{n,1},u^{1}\}\cup\{u^{n,1,j}|2\leq j\leq r-1\}$ for $n=r$.
Then $S_{i}$ induces a star $K_{1,r}$ with the center $u^{2i-1,2i}$ for $1\leq i\leq\frac{n-1}{2}$ and with the center $u^{n,1}$ for $i=\frac{n+1}{2}$ respectively (see Fig. 3(right)).

Let $S=\cup_{i=1}^{\frac{n+1}{2}}S_{i}$. Then $N_{Q_{n}}(u)\subseteq S$, and $u$ is an isolated vertex of $Q_{n}-S$. If $n\geq4$, then the vertex $\overline{u}$  belongs to $Q_{n}-S$, so the $S_{i}$'s for $1\leq i\leq \frac{n+1}{2}$ form a $K_{1,r}$-structure cut of $Q_n$. If $n=3$ and $r=2$, then $S=\{100,010,110\}\cup\{001,101,111\}$,  so $S$ forms a $K_{1,2}$-structure cut of $Q_{3}$ since $u^{2,3}=011$ belongs to $Q_{3}-S$. If $n=r=3$, then  $S=\{100,010,110,111\}\cup\{001,101,100,111\}$, so $S$ forms a $K_{1,3}$-structure cut of $Q_{3}$ since $u^{2,3}=011$ belongs to $Q_{3}-S$.

If $n\geq 4$ is even,
let $S_{i}=\{u^{2i-1},u^{2i},u^{2i-1,2i}\}\cup\{u^{2i-1,2i,2i+j}|1\leq j\leq r-2\}$ for $1\leq i\leq\frac{n}{2}$. Then $S_{i}$ induces a star $K_{1,r}$ with the center $u^{2i-1,2i}$ for $1\leq i\leq\frac{n}{2}$.
 Then  $u$ is an isolated vertex in $Q_{n}-S$ and $\overline{u}$ belongs to $Q_{n}-S$, where $S=\cup_{i=1}^{\frac{n}{2}}S_{i}$. Also $S$ forms a $K_{1,r}$-structure cut of $Q_n$.
\end{proof}
\end{Lem}

\begin{Rem}\rm{ Obviously $Q_2$ has no $K_{1,2}$-structure cut.}\end{Rem}
\begin{Rem}\rm{For the $K_{1,r}$-structure cut $S_i$'s, $1\leq i\leq \lceil\frac{n}{2}\rceil$, in the proof of  Lemma \ref{upper},   $S_{m}\cap S_k=\emptyset$ for each pair $1\leq m<k\leq \lfloor \frac{n}{2} \rfloor$, and $S_m\cap S_{\lceil\frac{n}{2}\rceil}\not=\emptyset$ for $1\leq m<\lceil\frac{n}{2}\rceil$ if and only if $m=1$ and $n=r\geq 3$ is odd (in this case, $S_{1}\cap S_{\frac{n+1}{2}}=\{u^{1},u^{1,2,n}\}$). We now give a proof as follows.   Recall that $S_{i}=\{u^{2i-1},u^{2i},u^{2i-1,2i}\}\cup \{u^{2i-1,2i,2i+j}|1\leq j\leq r-2\}$  for $1\leq i\leq\frac{n-1}{2}$.  For   $1\leq m<k\leq\lfloor \frac{n}{2} \rfloor$,   $\{2m-1,2m\}\cap\{2k-1,2k\}=\emptyset$, and thus  $\{2m-1,2m,2m+j_{1}\}\not=\{2k-1,2k,2k+j_{2}\}$ for $1\leq j_{1},j_{2}\leq r-2$, which implies that $S_{m}\cap S_k=\emptyset$.
Next suppose that $S_m\cap S_{\lceil\frac{n}{2}\rceil}\not=\emptyset$ for $1\leq m<\lceil\frac{n}{2}\rceil$. Then $n\geq 3$ is odd. If $n>r$, then $S_{\frac{n+1}{2}}=\{u^{n},u^{n,1}\}\cup\{u^{n,1,j}|2\leq j\leq r\}$. Since $1<2m <n$, there are $ 1\leq j_{1}\leq r-2$ and $2\leq j_{2}\leq r-1$ such that $\{2m-1,2m,2m+j_{1}\}=\{n,1,j_{2}\}$, which implies that  $2m+j_{1}=n$  and $2m-1=1$.  So $m=1$,  and $n=2+j_{1}\leq r$, a contradiction.
So we may assume that  $n=r\geq3$. Then $S_{\frac{n+1}{2}}=\{u^{n},u^{n,1},u^{1}\}\cup\{u^{n,1,j}|2\leq j\leq r-1\}$.  Similarly we have that $m=1$. Conversely, if  $m=1$ and $n=r\geq 3$ is odd, then we can find that $S_{1}\cap S_{\frac{n+1}{2}}=\{u^{1},u^{1,2,n}\}$. The proof is complete. }
\end{Rem}

\begin{figure}[!htbp]
\begin{center}
\includegraphics[totalheight=6.5cm]{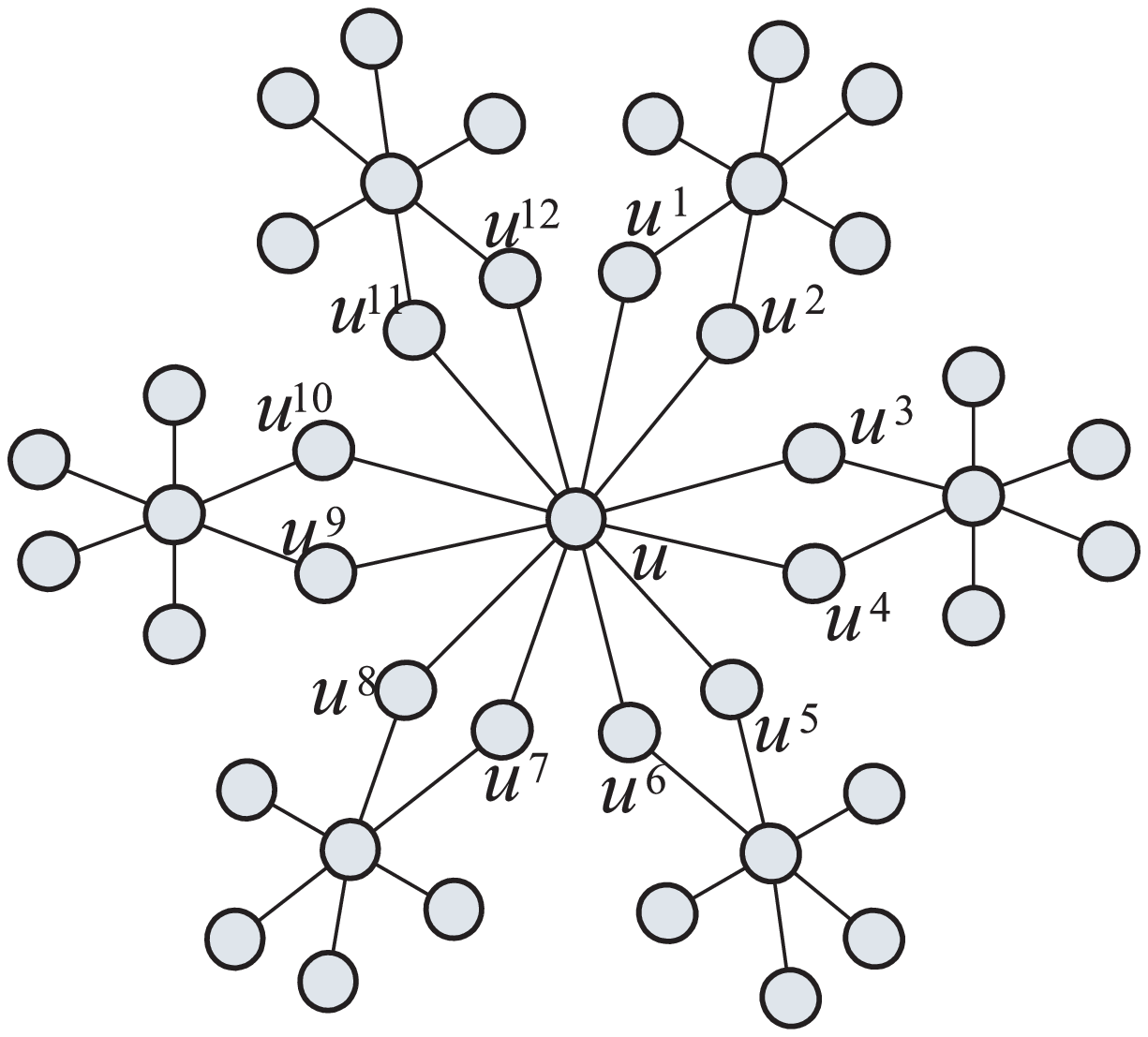}
\includegraphics[totalheight=6.5cm]{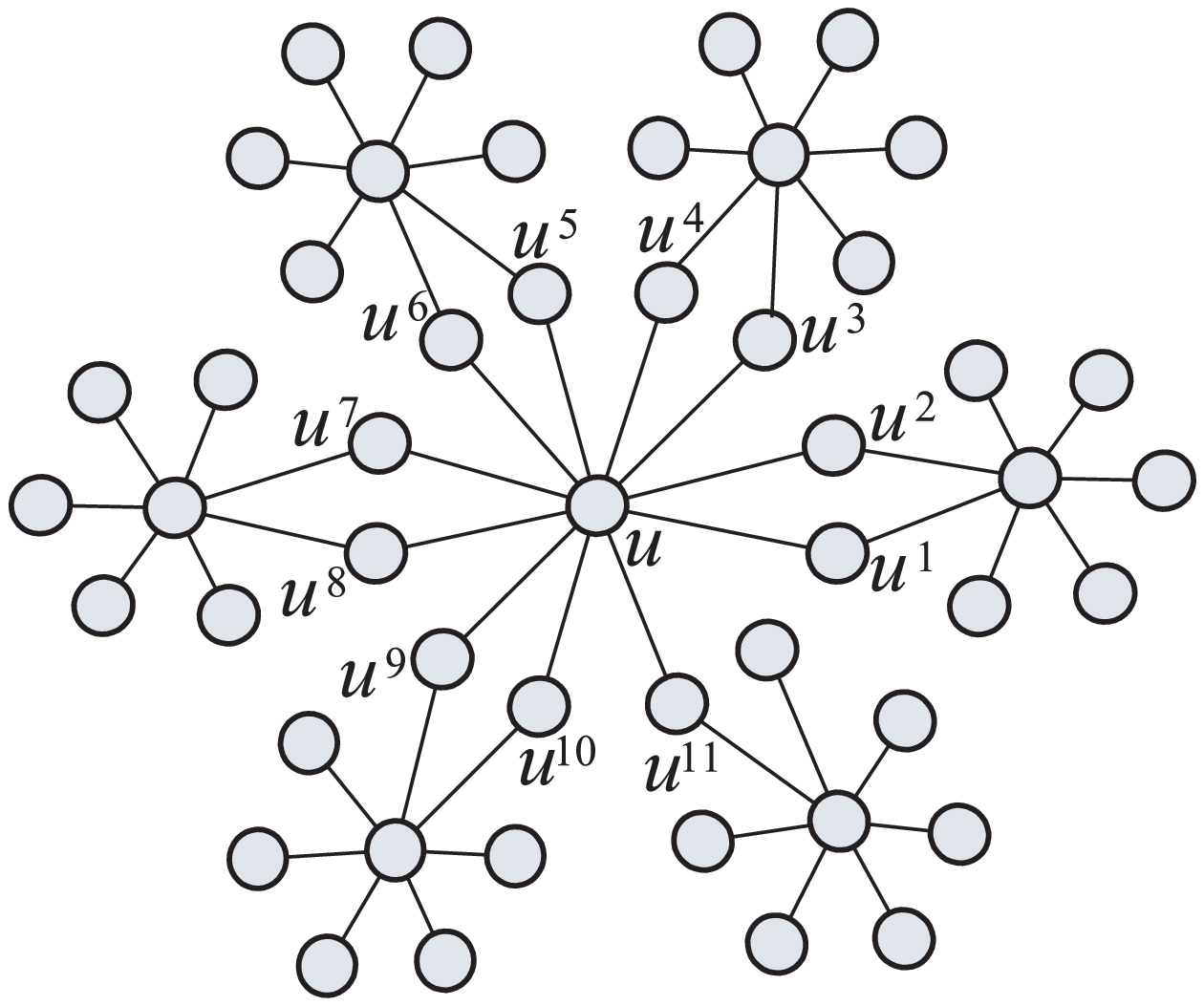}
\caption{\label{Q12 }\small{$K_{1,6}$-structure cut of $Q_{12}$ and  $K_{1,7}$-structure cut of $Q_{11}$.}}
\end{center}
\end{figure}

To describe our main result about $n$-hypercube $Q_n$, we define the  function $f(r)$ for all integers $r\geq 2$ as follows.

\begin{numcases}{f(r)=}
f_1(r)=
            {\rm max}\{\frac{r+7}{2},\frac{r^{2}+4r+3}{8}\},  &\text{if $r\geq3$ is odd}, \\
           f_2(r)= {\rm max}\{\frac{r^{2}+2r}{8},\frac{r+8}{2},\frac{r^{2}+6r+12}{12}\},   &\text{if $r\geq2$ is even}.
\end{numcases}

It is not difficult to find that $f_{1}(r)$ and $f_2(r)$ are both strictly  increasing functions for $r\geq 2$ by considering the property of  a quadratic function.  As Table \ref{tab:tab15} shows some initial values of  $f(r)$  for $2\leq  r \leq 20$, in general we can prove that $f(r)$ is an increasing function and integral except at $r=8$ in the  following lemma.
\begin{table}[h]
\centering
\caption{The values of $f(r)$ for $2\leq r \leq 20$.}\label{tab:tab15}
\begin{tabular}{|r|r|r|r|r|r|r|r|r|r|r|r|r|r|r|r|r|r|r|r|}
\hline
$r$ & $2$ & $3$ & $4$ & $5$ & $6$ & $7$ & $8$ & $9$ & $10$ & $11$ & $12$ & $13$ & $14$ & $15$ & $16$ & $17$ & $18$ & $19$ & $20$ \\
\hline
$f(r)$ & $5$ & $5$ & $6$ & $6$ & $7$ & $10$ & $\frac{31}{3}$ & $15$ & $15$ & $21$ & $21$ & $28$ & $28$ & $36$ & $36$  & $45$  & $45$ & $55$ & $55$\\
\hline
\end{tabular}
\qquad
\end{table}

\begin{Lem}
$f(r)$ is an increasing  function  for $r\geq 2$ and integral except at $r=8$, and  for odd  $r\geq 9$,
\begin{equation}\label{function1}
f(r)=f(r+1)=\frac{r^{2}+4r+3}{8}.
\end{equation}
\begin{proof}By Eq. (3.1), we know that
\[f_1(r)=
            \begin{cases}
            \frac{r+7}{2},  &\text{if $ r=3$},\\
            \frac{r^{2}+4r+3}{8},  &\text{if $r\geq5$ is odd}, \\
            \end{cases}\]
and by Eq. (3.2),
\[f_2(r)=
            \begin{cases}
            \frac{r+8}{2},  &\text{if $2\leq r\leq6$ is even},\\
            \frac{r^{2}+6r+12}{12},  &\text{if $r=8$ }, \\
            \frac{r^{2}+2r}{8},  &\text{if $r\geq10$ is even}.\\
            \end{cases}\]
So we have
$$
\aligned
f_{1}(r)=\frac{r^{2}+4r+3}{8}, ~~~~\mbox{for odd } r\geq 5, \mbox{ and}\\
f_{2}(r)=\frac{r^{2}+2r}{8}, ~~~~\mbox{for even }r\geq10.
\endaligned
$$
Further, if $r_{2}=r_{1}+1$, then
$$
\aligned
\frac{r_{1}^{2}+4r_{1}+3}{8}=\frac{r_{2}^{2}+2r_{2}}{8}.
\endaligned
$$
The above three equalities  imply that for odd $r\geq9$,  $f(r)=f(r+1)=\frac{r^{2}+4r+3}{8}$, so Eq. (\ref{function1}) holds. Together with Table \ref{tab:tab15} we know that   $f(r)$ is an  increasing function
for $r\geq 2$.

It remains to  prove that $f(r)$ is  integral for $r\geq 9$. Let $r+1=2k\geq 10$. Then $f(r)=f(r+1)=f_2(2k)=\frac{{(2k)}^{2}+2\times 2k}{8}=\frac{k(k+1)}{2}$, which is an integer.
\end{proof}
\end{Lem}
\begin{Lem}\label{F1}
For all integers $r\geq 2$, $r<f(r)$.
\end{Lem}
\begin{proof}
Obviously, we have
$$
\aligned
f_1(3)-3=\frac{3+7}{2}-3=2>0, \mbox{ and}
\endaligned
$$
$$
\aligned
f_1(r)-r=\frac{r^{2}+4r+3}{8}-r=\frac{1}{8}(r-1)(r-3)>0, \mbox{ for } r\geq5,
\endaligned
$$
which implies that for odd  $r\geq3$, the result holds.

For even $r\geq2$, we have
$$
f_2(r)-r=\frac{r+8}{2}-r=\frac{8-r}{2}>0, \mbox{ for } 2\leq r\leq6, \hspace{0.5cm}
f_2(8)-8=\frac{7}{3}>0, \mbox{ and}
$$
$$
f_2(r)-r=\frac{r^{2}+2r}{8}-r=\frac{1}{8}r(r-6)>0, \mbox{ for } r\geq10,
$$
so the result holds.
\end{proof}
\begin{Lem}\label{lower}
If integers $r\geq2$ and $n>f(r)$, then we have $\kappa(Q_{n};K_{1,r})\geq \lceil\frac{n}{2}\rceil$ and
$\kappa^{s}(Q_{n};K_{1,r})\geq \lceil\frac{n}{2}\rceil$.
\begin{proof}Since $\kappa^{s}(Q_{n};K_{1,r})\leq \kappa(Q_{n};K_{1,r})$, we only show
$\kappa^{s}(Q_{n};K_{1,r})\geq \lceil\frac{n}{2}\rceil$.
Suppose to the contrary that $\kappa^{s}(Q_{n};K_{1,r})< \lceil\frac{n}{2}\rceil$. Then $Q_{n}$ has a set $F$ of subgraphs that each is a star of at most $r$ leaves so that
$|F|\leq \lceil\frac{n}{2}\rceil-1$ and $Q_{n}-F$ is disconnected. So
\begin{equation}\label{F}
|V(F)|\leq(1+r)|F|\leq(1+r)(\lceil\frac{n}{2}\rceil-1)\leq\frac{1}{2}(r+1)(n-1).
\end{equation}

Let $C$ be a smallest component of $Q_{n}-F$ and $k=|V(C)|$. We distinguish the following three cases by considering the neighborhood of $C$ and $g$-extra connectivity in $Q_n$.

\textbf{Case $1$}. $k=1$.

 Let $C=\{u\}$. By Lemma 3.2, $|N_{Q_{n}}(u)\cap V(K_{1,r'})|\leq 2$ for  each member $K_{1,r'}$ in $F$, $0\leq r'\leq r$. Thus
\begin{eqnarray}\nonumber
n=|N_{Q_{n}}(u)|&\leq&\sum_{K\in F}|N_{Q_{n}}(u)\cap V(K)|\leq 2|F|\\
&\leq& 2(\lceil\frac{n}{2}\rceil-1)\leq2(\frac{n+1}{2}-1)
=n-1, \nonumber
\end{eqnarray}
 a contradiction.

\textbf{Case $2$}. $2\leq k \leq\frac{r}{2}+1$.

 From the given conditions, we know that $n\geq 6$. By Lemma 3.5, we have $|N_{Q_{n}}(C)|\geq nk-2(k-1)-\tbinom{k-1}{2}$. By Lemma 3.4, $|N_{Q_{n}}(C)\cap V(K_{1,r'})|\leq 2(k-1)$ for  each member $K_{1,r'}$ in $F$, $0\leq r'\leq r$.
 We have
$$
\aligned
nk-2(k-1)-\tbinom{k-1}{2}&\leq |N_{Q_{n}}(C)|\leq 2(k-1)|F|\leq 2(k-1)(\lceil\frac{n}{2}\rceil-1)\\
&\leq  (k-1)(n-1),
\endaligned
$$
which implies that
 $$n\leq \frac{k(k-1)}{2}.$$
If $r$ is even, then $n\leq \frac{r^{2}+2r}{8}$, contradicting $n>{\rm max}\{\frac{r^{2}+2r}{8},
\frac{r+8}{2},\frac{r^{2}+6r+12}{12}\}=f(r)$.
If $r$ is odd, then $n\leq \frac{r^{2}-1}{8}$, contradicting  $n>{\rm max}\{\frac{r+7}{2},\frac{r^{2}+4r+3}{8}\}=f(r)$.

\textbf{Case $3$}. $k\geq\frac{r+1}{2}+1$.

If $r$ is even,
then $ k\geq\frac{r}{2} +2$. Since $n>f(r)\geq\frac{r+8}{2}$, $0<\frac{r}{2}+1\leq n-4$, by Lemma 3.6 we have
\begin{eqnarray}
\kappa_{\frac{r}{2}+1}(Q_{n})&=&(2+\frac{r}{2})n-2 (\frac{r}{2}+1)-\tbinom{\frac{r}{2}+1}{2}\nonumber\\
&=&\frac{-r^{2}}{8}+\frac{rn}{2}+2n-\frac{5r}{4}-2. \label{k}
\end{eqnarray}

\noindent Since $Q_{n}-F$ is disconnected and $C$ is a smallest component of  $Q_{n}-F$, $|V(F)|\geq \kappa_{\frac{r}{2}+1}(Q_{n})$, so by Ineq. (\ref{F}) and Eq. (\ref{k}) we have
$$
\aligned
\frac{1}{2}(r+1)(n-1)\geq \frac{-r^{2}}{8}+\frac{rn}{2}+2n-\frac{5r}{4}-2,
\endaligned
$$
which implies that $n\leq\frac{r^{2}+6r+12}{12}$, contradicting  $n>{\rm max}\{\frac{r^{2}+2r}{8},\frac{r+8}{2},\frac{r^{2}+6r+12}{12}\}=f(r)$.

If $r$ is odd, then
$ k\geq\frac{r-1}{2}+2$. Since $n>f(r)\geq\frac{r+7}{2}$, $0<\frac{r-1}{2}+1\leq n-4$, by Lemma 3.6 we have
\begin{eqnarray}
\kappa_{\frac{r-1}{2}+1}(Q_{n})&=&(2+\frac{r-1}{2})n-2 (\frac{r-1}{2}+1)-\tbinom{\frac{r-1}{2}+1}{2}\nonumber\\
&=&\frac{-r^{2}}{8}+\frac{(r+3)n}{2}-r-\frac{7}{8}. \label{kodd}
\end{eqnarray}

Since $Q_{n}-F$ is disconnected and $C$ is a smallest component of  $Q_{n}-F$, $|V(F)|\geq\kappa_{\frac{r-1}{2}+1}(Q_{n})$. So by Ineq. (\ref{F}) and Eq. (\ref{kodd}) we have

$$
\aligned
\frac{1}{2}(r+1)(n-1)\geq\frac{-r^{2}}{8}+\frac{(r+3)n}{2}-r-\frac{7}{8},
\endaligned
$$
which implies  $n\leq\frac{r^{2}+4r+3}{8}$, a contradiction to $n>{\rm max}\{\frac{r+7}{2},\frac{r^{2}+4r+3}{8}\}=f(r)$.
\end{proof}
\end{Lem}

From Lemma \ref{F1} we know that the condition of Lemma \ref{lower} implies that of Lemma \ref{upper}. Hence
we obtain the following main result in this section.
\begin{The}\label{hypercube}
If $r\geq2$ and $n>f(r)$, then $\kappa(Q_{n};K_{1,r})=\kappa^{s}(Q_{n};K_{1,r})= \lceil\frac{n}{2}\rceil$.
\end{The}

\section{Some low dimensional cases of $Q_n$ }
For $r\leq 6$ and $Q_n$, the remaining cases not solved are to determine the values of $\kappa(Q_{n};K_{1,r})$ and $\kappa^{s}(Q_{n};K_{1,r})$   for $4\leq r \leq 6$ and $n=r$ and $r+1$.

In this section, we will solve separately the low dimensional cases, which cannot be treated in the previous unified way. Latifi \cite {5} express $Q_n=Q_n^0 \bigotimes Q_n^1$, where $Q_n^0\cong Q_{n-1}$ and $Q_n^1\cong Q_{n-1}$. $Q_n^0$ and $Q_n^1$ induced by the vertices with the $i$th coordinates $0$ and $1$ respectively, where $1\leq i\leq n$. In following, $N_{G-A}(A)=\{x|xy\in E(G), x\in G-A, y\in A\}$.
\begin{Lem}$\kappa^{s}(Q_4,K_{1,4})\geq2.$
\end{Lem}
\begin{proof} We set $F$ being a star of at most $4$ leaves, $F_i=F\cap Q_4^i$$(i=0,1)$. It is sufficient to prove that $Q_4-F$ is connected. Without loss of generality, assume the center of $F$ belongs to $Q_4^0$. It is noticed that $Q_4^0-F_0$ and $Q_4^1-F_1$ are both connected since $\kappa^{s}(Q_3,K_{1,3})=2$. Since there exists at least $V(Q_4^0-F_0)=2^3-4=4$ edges between $Q_4^0-F_0$ and $Q_4^1-F_1$ but $|V(F_1)|\leq1$, there is an edge between $Q_4^0-F_0$ and $Q_4^1-F_1$. Thus $Q_4-F$ is connected.
\end{proof}
\begin{Lem}$\kappa^{s}(Q_5,K_{1,4})\geq3.$
\end{Lem}
\begin{proof} We set $F_i$ be a star of at most $4$ leaves. It is sufficient to prove that $Q_5-F_1-F_2$ is connected. If $F_1\ncong K_{1,4}$ and $F_2\ncong K_{1,4}$, then the result holds since $\kappa^{s}(Q_5,K_{1,3})=3.$ Thus we assume that $F_i\cong K_{1,4}$ and $Q_5^i\cap F_2=F_2^i(i=0,1)$. Without loss of generality, assume $F_1\subseteq Q_5^0$. Since $\kappa^{s}(Q_4,K_{1,4})\geq2$ by Lemma 4.1, $Q_5^1-F_2^1$ is connected. If $Q_5^0-F_1-F_2^0$ is connected, then $Q_5-F_1-F_2$ is connected since $|V(Q_5^0-F_1-F_2^0)|\geq2^4-10=6$ and $|V(F_2^1)|\leq5$, there is a vertex in $Q_5^0-F_1-F_2^0$ which has a neighbor in $Q_5^1-F_2^1$.
If $Q_5^0-F_1-F_2^0$ is disconnected and each component of $Q_5^0-F_1-F_2^0$ connects to $Q_5^1-F_2^1$, then $Q_5-F_1-F_2$ is connected. Hence, we consider there is a component $C$ of $Q_5^0-F_1-F_2^0$ which is not connecting to $Q_5^1-F_2^1$. Then $N_{Q_5}(C)\subseteq (F_1\cup F_2)$ and $N_{Q_5^1}(C)\subseteq F_2^1$, which implies that $|V(C)|=|N_{Q_5^1}(C)|\leq |V(F_2^1)|\leq5$.
If $1\leq|V(C)|\leq2$, then $5\leq|N_{Q_5}(C)|\leq\sum_{i=1}^2|N_{Q_5}(C)\cap F_i|\leq 4$ by Lemmas 3.2 and 3.3, a contradiction.
If $|V(C)|=3$, then $10\leq|N_{Q_5}(C)|\leq\sum_{i=1}^2|N_{Q_5}(C)\cap F_i|\leq 8$ by Lemma 3.4, a contradiction.
If $4\leq|V(C)|\leq5$, then $|N_{Q_5}(C)|\geq 11$ by Lemma 3.5, so we have $11\leq|N_{Q_5}(C)|\leq\sum_{i=1}^2|V(F_i)|\leq 10$, a contradiction. Thus $Q_5-F_1-F_2$ is connected.
\end{proof}
By the above two lemmas and Lemma 3.7, we obtain the following theorem.
\begin{The}For $4\leq n\leq 5$, $\kappa(Q_n,K_{1,4})=\kappa^{s}(Q_n,K_{1,4})=\lceil\frac{n}{2}\rceil.$
\end{The}
\begin{Lem}$\kappa^{s}(Q_5,K_{1,5})\geq3.$
\end{Lem}
\begin{proof}Suppose to the contrary that $\kappa^{s}(Q_{5};K_{1,5})\leq2$. Let $F_i$ be a star of at most 5 leaves such that $Q_{5}-F_1-F_2$ is disconnected and $C$ be a smallest component of $Q_{5}-F_1-F_2$. We assume that for $i=0,1$, $F_1^i=F_1\cap Q_5^i$, $F_2^i=F_2\cap Q_5^i$ and $C_i=C\cap Q_5^i$. If $F_1\ncong K_{1,5}$ and $F_2\ncong K_{1,5}$, then $Q_5-F_1-F_2$ is connected since $\kappa^{s}(Q_5,K_{1,4})=3$ by Theorem 4.3. Thus we assume that $F_i\cong K_{1,5}$.

\textbf{Case 1.} $F_1\cong K_{1,5}$ and $F_2\ncong K_{1,5}$. We set $x$ is the center of $F_1$. Without loss of generality, assume that $F_2\subseteq Q_5^0$. Since $Q_5^1\cong Q_4$ and $\kappa^{s}(Q_4,K_{1,4})\geq2$ by Lemma 4.1, $Q_5^1-F_1^1$ is connected. If $Q_5^0-F_1^0-F_2$ is connected, then $C=Q_5-F_1-F_2$ since $|V(Q_5^0-F_1^0-F_2)|\geq2^4-10=6$ and $|V(F_1^1)|\leq5$, there is a vertex in $Q_5^0-F_1^0-F_2$ which has a neighbor in $Q_5^1-F_1^1$. If $Q_5^0-F_1^0-F_2$ is disconnected and each component of $Q_5^0-F_1-F_2^0$ connects to $Q_5^1-F_2^1$, then $C=Q_5-F_1-F_2$ is connected. Hence, there exists a smallest component $C'$ of $Q_5^0-F_1-F_2^0$ which is not connecting to $Q_5^1-F_1^1$. Then $N_{Q_5}(C')\subseteq (F_1\cup F_2)$ and $N_{Q_5^1}(C')\subseteq F_1^1$, which implies that $|V(C')|=|N_{Q_5^1}(C')|\leq |V(F_1^1)|\leq5$. As we know that $C'$ is also a component of $Q_5-F_1-F_2$. Since $Q_5^1-F_1^1$ is connected, the components of $Q_5-F_1-F_2$ either contains $Q_5^1-F_1^1$ or not. If $C\neq C'$, then $(Q_5^1-F_1^1)\subseteq C$ and $|V(C)|\geq|V(Q_5^1-F_1^1)|\geq2^4-5=11>|V(C')|$, it is a contradiction since $C$ is a smallest component. Thus $C=C'$. If $1\leq|V(C)|\leq2$, then $5\leq|N_{Q_5}(C)|\leq\sum_{i=1}^2|N_{Q_5}(C)\cap F_i|\leq 4$ by Lemmas 3.2 and 3.3, a contradiction. If $|V(C)|=3$, then $10\leq|N_{Q_5}(C)|\leq\sum_{i=1}^2|N_{Q_5}(C)\cap F_i|\leq 8$ by Lemma 3.4, a contradiction. If $4\leq|V(C)|\leq5$, then $|N_{Q_5}(C)|\geq 11$ by Lemma 3.5, so we have $11\leq|N_{Q_5}(C)|\leq\sum_{i=1}^2|V(F_i)|\leq 10$, a contradiction. Thus $Q_5-F_1-F_2$ is connected.

\textbf{Case 2.} $F_i\cong K_{1,5}$$(i=1,2)$. We set $F_1=\{x,x_1,x_2,x_3,x_4,x_5\}$, $F_2=\{y,y_1,y_2,y_3,y_4,y_5\}$ where $\{xx_i,yy_i\}\subset E(Q_5)(i\in\{1,2,\ldots,5\})$. We have the following cases by the positions of $x$ and $y$.

\textbf{Case 2.1.} Both $x$ and $y$ belong to $V(Q_5^i)$. Without loss of generality, assume that $x$ and $y$ belong to $V(Q_5^0)$ and $\{x_5,y_5\}\subset V(Q_5^1)$, Then $Q_5^{1}-x_5-y_5$ is connected since $\kappa(Q_5^1)=4$. Thus $C_1=Q_5^{1}-x_5-y_5$ or $C_1=\emptyset$. If $C_1=Q_5^{1}-x_5-y_5$, then $|V(C)|\geq|V(C_1)|=2^4-2=14$. We have $|V(Q_5-F_1-F_2)|-|V(C)|\leq2^5-12-14=6$, it is a contradiction since $C$ is a smallest component. Therefore $C=C_0$. Then $N_{Q_5^1}(C)\subseteq \{x_5,y_5\}$ and $|V(C)|=|N_{Q_5^1}(C)|\leq2$. If $1\leq|V(C)|\leq2$, then $5\leq|N_{Q_5}(C)|\leq \sum_{i=1}^2|N_{Q_5}(C)\cap F_i|\leq 4$ by Lemmas 3.2 and 3.3, a contradiction. Thus $Q_5-F_1-F_2$ is connected.

\textbf{Case 2.2.} Either $x$ or $y$ belongs to $V(Q_5^i)$. Without loss of generality, assume that $x\in V(Q_5^0)$, $y\in V(Q_5^1)$ and $x_5\in V(Q_5^1)$, $y_5\in V(Q_5^0)$. Then $N_{Q_5^0}(C_0)\subseteq (F_1^0\cup\{y_5\})$. If $1\leq|V(C_0)|\leq2$, then $4\leq|N_{Q_5^0}(C_0)|\leq |N_{Q_5^0}(C_0)\cap F_1^0|+1\leq 3$ by Lemmas 3.2 and 3.3, a contradiction. If $3\leq|V(C_0)|\leq5$, then $6\leq|N_{Q_5^0}(C_0)|\leq |\{x_1,x_2,x_3,x_4,y_5\}|=5$ by Lemmas 3.5, a contradiction.
Thus $|V(C_0)|\geq 6$ and $|V(C_1)|\geq 6$ by a similar argument. So $|V(C)|\geq12$ and $|V(Q_5-F_1-F_2)|-|V(C)|\leq 2^5-12-12=8$. It contradicts to that $C$ is a smallest component. Thus $Q_5-F_1-F_2$ is connected.
\end{proof}
\begin{Lem}For $5\leq r\leq6$, $\kappa^{s}(Q_6,K_{1,r})\geq3$.
\end{Lem}
\begin{proof}Suppose to the contrary that $\kappa^{s}(Q_{6};K_{1,r})\leq2$. Let $F_i(i=1,2)$ be a star of at most $r$ leaves and $Q_{6}-F_1-F_2$ is disconnected. Let $C$ be a smallest components of $Q_{6}-F_1-F_2$.

If $1\leq|V(C)|\leq2$, by Lemmas 3.2 and 3.3, $|N_{Q_{6}}(C)\cap V(F_i)|\leq 2$. Thus
$6\leq|N_{Q_{6}}(C)|\leq\sum_{i=1}^2|N_{Q_{6}}(C)\cap V(F_i)|\leq 4$
 a contradiction.

If $|V(C)|=3$, then $|N_{Q_{6}}(C)\cap V(F_i)|\leq 4$ by Lemma 3.4 and $|N_{Q_{6}}(C)|\geq 13$ by Lemma 3.5. Thus
$13\leq|N_{Q_{6}}(C)|\leq\sum_{i=1}^2|N_{Q_{6}}(C)\cap V(F_i)|\leq 8$,
 a contradiction.

If $|V(C)|\geq4$, by Lemma 3.6 we have $\kappa_3(Q_6)=15$. Since $Q_{6}-F_1-F_2$ is disconnected and $C$ is a smallest component of $Q_{6}-F_1-F_2$, $14\geq|V(F_1)|+|V(F_2)|\geq\kappa_3(Q_6)=15$.
A contradiction.
\end{proof}
By Lemmas 4.4 and 4.5, we have $\kappa^{s}(Q_n,K_{1,5})\geq 3$ with $5\leq n\leq 6$. Thus we get the following theorem by Lemma 3.7.
\begin{The}For $5\leq n\leq 6$, $\kappa(Q_n,K_{1,5})=\kappa^{s}(Q_5,K_{1,5})=3.$
\end{The}
\begin{Lem}$\kappa^{s}(Q_7,K_{1,6})\geq4.$
\end{Lem}
\begin{proof}Suppose to the contrary that $\kappa^{s}(Q_{7};K_{1,6})\leq3$. Let $F_i$ be star of at most 6 leaves such that $Q_{7}-\cup_{i=1}^3F_i$ is disconnected and $C$ be a smallest components of $Q_{7}-\cup_{i=1}^3F_i$ with $|V(C)|=g+1$. Then $N_{Q_7}(C)\subseteq\cup_{i=1}^3 F_i$ and $N_{Q_7}(C)\subseteq\cup_{i=1}^3 (N_{Q_7}(C)\cap F_i)$. By Lemma 3.5, we have
$$
\aligned
7(g+1)-2g-\frac{1}{2}g(g-1)\leq|N_{Q_7}(C)|\leq\sum_{i=1}^3|F_i|\leq21,
\endaligned
$$
it implies that $g\leq4$ or $g\geq7$; for $g\leq4$, it contradicts to $|N_{Q_7}(C)|\leq\sum_{i=1}^3|N_{Q_7}(C)\cap F_i|$ by Lemmas 3.2 and 3.4.
Thus we have $|V(C)|\geq 8$.
We assume that $C_i=C\cap Q_7^i$$(i=0,1)$. If $F_i\ncong K_{1,6}$$(i=1,2,3)$, then $Q_{7}-\cup_{i=1}^3F_i$ is connected since $\kappa^{s}(Q_7,K_{1,5})=4$ by Theorem \ref{hypercube}. So we consider $F_i\cong K_{1,6}$. Without loss of generality, we assume that $F_1\cong K_{1,6}$, $F_1\subseteq Q_7^0$ and $Q_7^i\cap F_2=F_2^i$, $Q_7^i\cap F_3=F_3^i$$(i=0,1)$. Since $Q_7^i\cong Q_6$ and $\kappa^{s}(Q_6,K_{1,6})\geq3$ by Lemma 4.5, $Q_7^1-F_2^1-F_3^1$ is connected. Thus we know the components of $Q_7-\cup_{i=1}^3F_i$ either contains $Q_7^1-F_2^1-F_3^1$ or not. We have $C_1=Q_7^1-F_2^1-F_3^1$ or $C_1=\emptyset$.
If $C_1=Q_7^1-F_2^1-F_3^1$, then $|V(C_1)|\geq2^6-14=50$. Since each vertex in $C_1$ has exactly one neighbor in $N_{Q_7^0}(C_1)$, $|N_{Q_7^0}(C_1)|=|V(C_1)|\geq50$ and $N_{Q_7^0}(C_1)\subseteq (F_1\cup F_2^0\cup F_3^0\cup C_0)$. We know $|N_{Q_7^0}(C_1)|\leq|V(F_1\cup F_2^0\cup F_3^0)|+|V(C_0)|\leq21+|V(C_0)|$, which implies that $|V(C_0)|\geq29$. We get $|V(C)|=|V(C_1)|+|V(C_0)|\geq79$ and $|V(Q_7-\cup_{i=1}^3F_i)|-|V(C)|<2^7-79=49$, it is a contradiction since $C$ is a smallest component. Therefore $C\subseteq Q_7^0-F_1-F_2^0-F_3^0$ with $|V(C)|\geq8$. We have $8\leq|V(C)|=|N_{Q_7^1}(C)|\leq|V(F_2^1\cup F_3^1)|$, then $|V(F_2^1\cup F_3^1)|\geq8$. By the above analysis, we know the component of $Q_7-\cup_{i=1}^3F_i$ which contains $Q_7^1-F_2^1-F_3^1$ have at least 29 vertices in $Q_7^0-F_1-F_2^0-F_3^0$, so
the component of $Q_7^0-F_1-F_2^0-F_3^0$ has at least 8 vertices.
Then $\kappa_6(Q_7^0)\leq|V(F_1\cup F_2^0\cup F_3^0)|\leq|V(\cup_{i=1}^3F_i)|-|V(F_2^1\cup F_3^1)|\leq13$. As we know $\kappa_6(Q_6)=15$ by Lemma 3.6. It is a contradiction. Thus $Q_7-\cup_{i=1}^3F_i$ is connected.
\end{proof}
By Lemma 3.7 and Lemmas 4.5, 4.7, we have the following theorem.
\begin{The} $\kappa(Q_6,K_{1,6})=\kappa^{s}(Q_6,K_{1,6})=3$;
$\kappa(Q_7,K_{1,6})=\kappa^{s}(Q_7,K_{1,6})=4.$
\end{The}
\section{The star-Structure connectivity of Fold Hypercube }
In this section, we study the $\kappa(FQ_{n};K_{1,r})$ and $\kappa^{s}(FQ_{n};K_{1,r})$ for $r\geq2$.
It is known that $FQ_{n}$ is triangle-free for $n\geq3$.
\begin{Lem}{\rm\cite {16}}
Any two vertices in $FQ_{n}$ exactly have two common neighbors for $n\geq 4$ if they have any.
\end{Lem}

It is easy to find the above lemma is true when $n=2$. For $n=3$, we find
$N_{FQ_{n}}(011)\cap N_{FQ_{n}}(110)=\{010,001,100,111\}$, so Lemma 5.1 is fault when $n=3$.

\begin{Lem}{\rm\cite {11}}
Let $FQ_{n}$ be a folded hypercube. Then
\begin{enumerate}[i)]
            \item \label{cond 1}
             $\kappa(FQ_{n})=n+1$;

            \item \label{cond 2}
              $FQ_{n}$ is a bipartite graph if and only if $n$ is odd;

            \item \label{cond 3}
              If $FQ_{n}$ contains an odd cycle, then a shortest odd cycle has the length  $n+1$.
         \end{enumerate}

\end{Lem}

\begin{Lem}
Let $K_{1,r}$ be a star in $FQ_{n}$ with $n\geq 4$ and $n+1\geq  r\geq2$. If $u$ is a vertex in $FQ_{n}-K_{1,r}$, then
$|N_{FQ_{n}}(u)\cap V(K_{1,r})|\leq 2$, and equality holds if and only if $u$ is adjacent to exactly two leaves of $K_{1,r}$.
\begin{proof}Since $FQ_{n}$ is triangle-free for $n\geq 4$, $u$ cannot be adjacent to both a leaf and the center of $K_{1,r}$. If $u$ is just adjacent to a leaf of $K_{1,r}$, then by Lemma 5.1 $u$ has at most two neighbors in the leaves of $K_{1,r}$, which are adjacent to the center of $K_{1,r}$.
\end{proof}
\end{Lem}
For $n\geq 5$, $FQ_{n}$  has no 5-cycle or 3-cycle. So we can derive the following result by analogous arguments as Lemma \ref{connection}.
\begin{Lem}
Let $K_{1,r}$ be a star in $FQ_{n}$ with $n\geq 5$ and $n+1\geq r\geq2$. If $C$ is a connected subgraph in $FQ_{n}-K_{1,r}$ with $|V(C)|=k\geq2$, then
$|N_{FQ_{n}}(C)\cap V(K_{1,r})|\leq 2(k-1)$, and equality holds only  if $C$ is a star in $FQ_{n}$.
\end{Lem}
\begin{proof}Let $V(K_{1,r})=\{x,x_{1},x_{2},\ldots, x_{r}\}$ and $E(K_{1,r})=\{xx_{i}|1\leq i\leq r\}$. Then $x$ is the center of $K_{1,r}$. Let $V(C)=\{u_{1},u_{2},\ldots,u_{k}\}$.

First, we prove that $|N_{FQ_{n}}(C)\cap V(K_{1,r})|\leq 2(k-1)$.
Suppose to the contrary that $|N_{FQ_{n}}(C)\cap V(K_{1,r})|\geq 2(k-1)+1=2k-1$.
By Lemma $5.3$, each vertex $u_i$ in $C$ has at most 2 neighbors in $K_{1,r}$, and if $|N_{FQ_{n}}(u_{i})\cap V(K_{1,r})|=2$, then $u_{i}$ is adjacent to two leaves in $K_{1,r}$, so $2k\geq|N_{FQ_{n}}(C)\cap V(K_{1,r})|\geq 2k-1$. It means that there exists at least $k-1$ vertices in $C$ which each has two neighbors in $K_{1,r}$, and such neighbors are pairwise distinct. Without loss of generality, we assume
$\{u_{i}x_{2i-3},u_{i}x_{2i-2}\}\subset E(FQ_{n})$ for $2\leq i\leq k$. Since $C$ is connected, $u_{1}$ is adjacent to $u_{i}$ for some $2\leq i\leq k$. If $u_{1}x_{j}\in E(FQ_{n})$ with  $2k-1\leq j\leq r$, then there exists a 5-cycle $u_{1}x_{j}xx_{2i-3}u_{i}u_{1}$, a contradiction. Otherwise, $u_{1}x\in E(FQ_{n})$. Then
 $N_{FQ_{n}}(u_{i})\cap N_{FQ_{n}}(x)=\{x_{2i-3},x_{2i-2},u_{1}\}$, contradicting  Lemma 5.1. Hence
 $|N_{FQ_{n}}(C)\cap V(K_{1,r})|\leq 2(k-1)$.

 Next we show that if $|N_{FQ_{n}}(C)\cap V(K_{1,r})|= 2(k-1)$, then $C$ is a star in $FQ_{n}$. Suppose to the contrary that $C$ is not a star in  $FQ_{n}$. Then there exists a 4-vertex path $P_{4}$ in $C$, so  $4\leq k$
 and $6\leq|N_{FQ_{n}}(P_{4})\cap V(K_{1,r})|\leq8$ by Lemma 5.3. However,
 any two consecutive vertices in $P_{4}$ have at most two neighbors in $V(K_{1,r})$, since $FQ_{n}(n\geq5)$ has no 5-cycle and $|N_{FQ_{n}}(u)\cap N_{FQ_{n}}(x)|\leq2$ for $u\in V(P_{4})$. This implies that $P_{4}$ has at most four  neighbors in $V(K_{1,r})$, a contradiction.
\end{proof}
 The following lemma can be obtained from Theorem 2.11 of {\rm \cite {15}}.
%The proof of Lemma 4.4 is similar to Lemma 3.4.
\begin{Lem}{\rm \cite {15}}
Let $C$ be a subgraph of $FQ_{n}$ with $|V(C)|=g+1$,
 for $n\geq 5$, $1\leq g\leq n+2$. Then $|N_{FQ_{n}}(C)|\geq (n+1)(g+1)-2g-\tbinom{g}{2}$.
\end{Lem}
\begin{Lem}{\rm\cite {14}}
For $n\geq 7$,
\[\kappa_{g}(FQ_{n})=
            \begin{cases}
            (g+1)(n+1)-2g-\tbinom{g}{2},  &\text{if \hspace{0.1cm} $0\leq g\leq n-3$}, \\
            \frac{n(n+1)}{2},   &\text{if \hspace{0.1cm} $n-2\leq g\leq n+1$}. \\
            \end{cases}\]
\end{Lem}
\begin{Lem}\label{FU}
For $n+1\geq r \geq 2$, $n\geq3$, $\kappa(FQ_{n};K_{1,r})\leq \lceil\frac{n+1}{2}\rceil$ and
$\kappa^{s}(FQ_{n};K_{1,r})\leq \lceil\frac{n+1}{2}\rceil$.
\end{Lem}
\begin{proof}Since $\kappa^{s}(FQ_{n};K_{1,r})\leq \kappa(FQ_{n};K_{1,r})$, we only prove
$\kappa(FQ_{n};K_{1,r})\leq \lceil\frac{n+1}{2}\rceil$.
 Let $u=000\cdots 0$ be a vertex in $Q_{n}$. Then $N_{FQ_{n}}(u)=\{\overline{u}\}\cup\{u^{i}|1\leq i\leq n\}$.

\textbf{Case 1}. $n\geq3$ is odd. For  $1\leq i\leq\frac{n-1}{2}$, let $S_{i}=\{u^{2i-1},u^{2i},u^{2i-1,2i}\}\cup \{u^{2i-1,2i,2i+j}|1\leq j\leq r-2\}$ with $r\leq n$ and $S_{i}=\{u^{2i-1},u^{2i},u^{2i-1,2i}\}\cup \{u^{2i-1,2i,2i+j}|1\leq j\leq n-2\}\cup \{\overline{u}^{2i-1,2i}\}$ with $r=n+1$. Let
$S_{\frac{n+1}{2}}=\{u^{n},\overline{u},\overline{u}^{n}\}\cup \{\overline{u}^{n,j}|1\leq j\leq r-2\}$. Noting that $\overline{u}^{n}=\overline{u^{n}}$, we also know that
$S_{i}$ induces a star $K_{1,r}$ with the center $u^{2i-1,2i}$ for $1\leq i\leq\frac{n-1}{2}$ and with the center $\overline{u}^{n}$ for $i=\frac{n+1}{2}$ respectively.
Let $S=\cup_{i=1}^{\frac{n+1}{2}}S_{i}$. Then $N_{FQ_{n}}(u)\subseteq S$, and $u$ is an isolated vertex of $FQ_{n}-S$. We can  see that  vertex ${u}^{1,n}\notin S_{i}$ for each $1\leq i\leq\frac{n+1}{2}$. So the $S_{i}'s$ for $1\leq i\leq\frac{n+1}{2}$ form a $K_{1,r}$-structure cut of $FQ_{n}$.

\textbf{Case 2}. $n\geq4$ is even. For $r\leq n$, let $S_{i}=\{u^{2i-1},u^{2i},u^{2i-1,2i}\}\cup\{u^{2i-1,2i,2i+j}|1\leq j\leq r-2\}$ when $1\leq i\leq\frac{n}{2}$ and $S_{\frac{n+2}{2}}=\{\overline{u}\}\cup\{\overline{u}^{j}|1\leq j\leq r\}$.
For $r=n+1$, let $S_{i}=\{u^{2i-1},u^{2i},u^{2i-1,2i}\}\cup\{u^{2i-1,2i,2i+j}|1\leq j\leq n-2\}\cup\{\overline{u}^{2i-1,2i}\}$ when $1\leq i\leq\frac{n}{2}$ and $S_{\frac{n+2}{2}}=\{\overline{u}, u^{1}, \overline{u}^{1}\}\cup\{\overline{u}^{1,j}|2\leq j\leq n\}$.
 Then $S_{i}$ induces a star $K_{1,r}$ with the center $u^{2i-1,2i}$ for $1\leq i\leq\frac{n}{2}$, and $S_{\frac{n+2}{2}}$ also induces a star $K_{1,r}$ with the center $\overline{u}$ for $r\leq n$ and with the center $\overline{u}^{1}$ for $r=n+1$ respectively. Let $S=\cup_{i=1}^{\frac{n+2}{2}}S_{i}$. We can see that  $u$ is an isolated vertex in $FQ_{n}-S$ and $u^{1,n}$ belongs to $FQ_{n}-S$. So  $S$ forms a $K_{1,r}$-structure cut of $FQ_{n}$.
\end{proof}

\begin{Rem}\rm{For the $K_{1,r}$-structure cut $S_i$'s, $1\leq i\leq\lceil\frac{n+1}{2}\rceil$, in the proof of Lemma 5.7,  any pair of distinct  $S_i$ and $S_j$ are disjoint for $n\geq 6$ and $r\leq n$. A proof is presented here. Recall that $S_{i}=\{u^{2i-1},u^{2i},u^{2i-1,2i}\}\cup\{u^{2i-1,2i,2i+j}|1\leq j\leq r-2\}$ for $1\leq i\leq\lfloor\frac{n}{2}\rfloor$. For $1\leq m<k\leq \lfloor\frac{n}{2}\rfloor$, $\{2m-1,2m\}\cap\{2k-1,2k\}=\emptyset$, and thus $\{2m-1,2m,2m+j_{1}\}\neq\{2k-1,2k,2k+j_{2}\}$ for $1\leq j_{1}, j_{2}\leq r-2$, which implies that $S_{m}\cap S_{k}=\emptyset$. We now consider   $S_{m}$ and $ S_{\lceil\frac{n+1}{2}\rceil}$ for $1\leq m\leq\lfloor\frac{n}{2}\rfloor$. If $n$ is odd, $S_{\frac{n+1}{2}}=\{u^{n},\overline{u},\overline{u}^{n},\overline{u}^{n,j}|
1\leq j\leq r-2\}$, and $1<2m<n$.  Since $\overline{u}^{n,j}$ agrees with  $u$ in exactly 2 positions, and $u^{2m-1,2m,2m+j_{1}}$ agrees with $u$ in exactly $n-3$ positions,  $\overline{u}^{n,j}\not=u^{2m-1,2m,2m+j_{1}}$ for $n>5$. So $S_{m}\cap S_{\lceil\frac{n+1}{2}\rceil}=\emptyset.$
If $n$ is even, $S_{\frac{n}{2}+1}=\{\overline{u},\overline{u}^{j}|1\leq j\leq r\}$, and   $S_{m}\cap S_{\frac{n}{2}+1}=\emptyset$  for $n\geq 6$.

For $r=n+1$ and $n\ge 6$, we have a unique pair of intersecting $K_{1,r}$-stars in the $S_i$'s, that is,  $S_{1}\cap S_{\frac{n+2}{2}}=\{u^{1}, \overline{u}^{1,2}\}$.  For $3\leq n\leq 5$, however, there are many pairs of intersecting $K_{1,r}$-stars in the $S_i$'s.}
\end{Rem}

In order to describe our main result about $n$-dimensional folded hypercube $FQ_n$, we define the  function $g(r)$ as follows.
\begin{numcases}{g(r)=}
g_1(r)={\rm max}\{6,\frac{r+5}{2},\frac{r^{2}+4r-5}{8}\},  &\text{if $r\geq3$ is odd}, \\
g_2(r)={\rm max}\{6,\frac{r^{2}+2r-8}{8},\frac{r+6}{2},\frac{r^{2}+6r}{12}\},   &\text{if $r\geq2$ is even}.
\end{numcases}
We find that $g_1(r)$ and $g_2(r)$ are both  increasing functions. Table \ref{tab:tab20} lists the values of $g(r)$ for $2\leq r \leq 20$.
We also have  the following monotonicity and integrality of function $g(r)$.
\begin{table}[h]
\centering
\caption{The values of $g(r)$ for $2\leq r \leq 20$.}\label{tab:tab20}
\begin{tabular}{|r|r|r|r|r|r|r|r|r|r|r|r|r|r|r|r|r|r|r|r|}
\hline
$r$ & $2$ & $3$ & $4$ & $5$ & $6$ & $7$ & $8$ & $9$ & $10$ & $11$ & $12$ & $13$ & $14$ & $15$ & $16$ & $17$ & $18$ & $19$ & $20$ \\
\hline
$g(r)$ & $6$ & $6$ & $6$ & $6$ & $6$ & $9$ & $\frac{28}{3}$ & $14$ & $14$ & $20$ & $20$ & $27$ & $27$ & $35$ & $35$ & $44$ & $44$ & $54$ & $54$ \\
\hline
\end{tabular}
%\qquad
\end{table}

\begin{Lem}
$g(r)$ is an increasing function for $r\geq 2$ and integral except at $r=8$, and for odd $r\geq 9$,
\begin{equation}\label{function2}
g(r)=g(r+1)=\frac{r^{2}+4r-5}{8}.
\end{equation}
\begin{proof}
By Eq. (5.1), we find that
\[g_1(r)=
            \begin{cases}
            6,  &\text{if $3\leq r\leq5$ is odd},\\
      \frac{r^{2}+4r-5}{8},  &\text{if $r\geq7$ is odd};\\
            \end{cases}\]
and by Eq. (5.2),
\[g_2(r)=
            \begin{cases}
            6,  &\text{if $2\leq r\leq6$ is even},\\
            \frac{r^{2}+6r}{12},  &\text{if $r=8$ },\\
      \frac{r^{2}+2r-8}{8},  &\text{if $r\geq10$ is even}.\\
            \end{cases}\]
So, we have
$$
\aligned
g_{1}(r)=\frac{r^{2}+4r-5}{8}, ~~~~\mbox{for odd}~~r\geq7,\\
g_{2}(r)=\frac{r^{2}+2r-8}{8}, ~~~~\mbox{for even}~~r\geq10.
\endaligned
$$
Moreover, if $r_{1}\geq9$ and $r_{2}=r_{1}+1$, then
$$
\aligned
\frac{r_{1}^{2}+4r_{1}-5}{8}=\frac{r_{2}^{2}+2r_{2}-8}{8},
\endaligned
$$
which means that for odd $r\geq9$, Eq. (\ref{function2}) holds. Together with Table \ref{tab:tab20}, we know that $g(r)$ is a
monotonically increasing function for $r\geq2$.

We now only prove that $g(r)$ is integral for $r\geq 9$. Let $r+1=2k\geq10$.  Then $g(r)=g(r+1)=g_2(2k)=\frac{(2k)^{2}+2(2k)-8}{8}=\frac{(k+2)(k-1)}{2}$, which is an integer.
\end{proof}
\end{Lem}
\begin{Lem}\label{G1}
For all integers $r\geq2$, $g(r)\geq r$.
\end{Lem}
\begin{proof}
We have
$$
\aligned
g_1(r)-r=6-r>0, \mbox{ for } 3\leq r\leq5;
\endaligned
$$
$$
\aligned
g_1(r)-r=\frac{r^{2}+4r-5}{8}-r=\frac{1}{8}(r+1)(r-5)>0, \mbox{ for } r\geq7.
\endaligned
$$
Therefore, if $r\geq3$ is odd , then $g_1(r)-r>0$.

On the other hand,
$$
\aligned
g_2(r)-r=6-r\geq0, \mbox{ for } 2\leq r\leq6;
\endaligned
$$
$$
\aligned
g_2(8)-8=\frac{8^{2}+48}{12}-8=\frac{4}{3}>0;
\endaligned
$$
$$
\aligned
g_2(r)-r=\frac{r^{2}+2r-8}{8}-r=\frac{1}{8}(r^{2}-6r-8)>0, \mbox{ for } r\geq10.
\endaligned
$$
So, if $r\geq2$ is even, then $g_2(r)-r\geq0$.
\end{proof}
\begin{Lem}\label{FL}
If integers $r\geq2$ and $n>g(r)$, then we have $\kappa(FQ_{n};K_{1,r})\geq \lceil\frac{n+1}{2}\rceil$ and
$\kappa^{s}(FQ_{n};K_{1,r})\geq \lceil\frac{n+1}{2}\rceil$.
\begin{proof}Since $\kappa^{s}(FQ_{n};K_{1,r})\leq\kappa(FQ_{n};K_{1,r})$, it suffices to show that
$\kappa^{s}(FQ_{n};K_{1,r})\geq \lceil\frac{n+1}{2}\rceil$.
Suppose to the contrary that $\kappa^{s}(FQ_{n};K_{1,r})$$<\lceil\frac{n+1}{2}\rceil$. Then there are  a set $F$ of subgraphs of $FQ_{n}$ that each is a star of at most $r$ leaves so that
$|F|\leq\lceil\frac{n+1}{2}\rceil-1$ and $FQ_{n}-F$ is disconnected. Let $C$ be a smallest component of $FQ_{n}-F$ and $k:=|V(C)|$. We consider following three cases.

\textbf{Case $1$}. $k=1$.

 Let $C=\{u\}$.  By Lemma 5.3, $|N_{FQ_{n}}(u)\cap V(K_{1,r'})|\leq 2$ for each member $K_{1,r'}$ in $F$, $0\leq r'\leq r$. Thus
\begin{eqnarray}\nonumber
n+1=|N_{FQ_{n}}(u)|&\leq&\sum_{K\in F}|N_{FQ_{n}}(u)\cap V(K)|\leq 2|F|\\
&\leq& 2(\lceil\frac{n+1}{2}\rceil-1)\leq2(\frac{n+2}{2}-1)
=n, \nonumber
\end{eqnarray}
a contradiction.

\textbf{Case $2$}. $2\leq k \leq \frac{r}{2} +1$.

Since $n>g(r)\geq\frac{r+5}{2}$, $2\leq k\leq\frac{r}{2}+1\leq n-2$. For $n\geq7$, by Lemma 5.5, we have $|N_{FQ_{n}}(C)|\geq (n+1)k-2(k-1)-\tbinom{k-1}{2}$. By Lemma 5.4, $|N_{FQ_{n}}(C)\cap V(K_{1,r'})|\leq 2(k-1)$ for each member $K_{1,r'}$ in $F$, $0\leq r' \leq r$. We have
$$
\aligned
(n+1)k-2(k-1)-\tbinom{k-1}{2}&\leq |N_{FQ_{n}}(C)|\leq 2(k-1)|F|
\leq2(k-1)(\lceil\frac{n+1}{2}\rceil-1)\\
&\leq(k-1)n,
\endaligned
$$
which implies that $n\leq \frac{(k-2)(k+1)}{2}$.
If $r$ is even, then $n\leq \frac{r^{2}+2r-8}{8}$, contradicting $n>{\rm max}\{6,\frac{r^{2}+2r-8}{8}, \frac{r+6}{2},\frac{r^{2}+6r}{12}\}=g(r)$. If $r$ is odd, then $n\leq \frac{r^{2}-9}{8}<\frac{r^{2}+4r-5}{8}\leq g(r)$, a contradiction.

\textbf{Case $3$}. $k\geq \frac{r+1}{2}+1$.

If $r$ is even, then
$ k\geq \frac{r}{2} +2$. Since $n>g(r)\geq{\rm max}\{
6,\frac{r+6}{2}\}$, $2\leq\frac{r}{2}+1\leq n-3$, by Lemma 5.6 we have
$$
\aligned
\kappa_{\frac{r}{2}+1}(FQ_{n})&=(\frac{r}{2}+2)(n+1)- 2(\frac{r}{2}+1)-\tbinom{\frac{r}{2}+1}{2}\\
&=\frac{-r^{2}}{8}+\frac{rn}{2}+2n-\frac{3r}{4}.
\endaligned
$$

 We also have that
\begin{equation}\label{FF}
|V(F)|\leq(1+r)|F|\leq (1+r)(\lceil\frac{n+1}{2}\rceil-1)\leq
\frac{1}{2}(r+1)n.
\end{equation}

 Since $FQ_{n}-F$ is disconnected, and $C$ is a smallest component of $FQ_{n}-F$ and $ |C|=k\geq \frac{r}{2} +2$, we have  $|V(F)|\geq\kappa_{\frac{r}{2}+1}(FQ_{n})$, so
$$
\aligned
\frac{1}{2}(r+1)n\geq\frac{-r^{2}}{8}+\frac{rn}{2}+2n-\frac{3r}{4},
\endaligned
$$

 \noindent which implies that $n\leq\frac{r^{2}+6r}{12}$, contradicting $n>{\rm max}\{6,\frac{r^{2}+2r-8}{8},
\frac{r+6}{2},\frac{r^{2}+6r}{12}\}=g(r)$.

If $r$ is odd and $n>g(r)\geq{\rm max}\{
6,\frac{r+5}{2}\}$, then $2\leq\frac{r+1}{2}\leq n-3$. By Lemma 5.6 we have that
\begin{eqnarray}
\kappa_{\frac{r+1}{2}}(FQ_{n})&=&(\frac{r+1}{2}+1)(n+1)-2 (\frac{r+1}{2})-\tbinom{\frac{r+1}{2}}{2}\nonumber\\
 &=&\frac{-r^{2}}{8}+\frac{(r+3)n}{2}-\frac{r}{2}+\frac{5}{8}\label{Fk}.
\end{eqnarray}
\noindent Since $FQ_{n}-F$ is disconnected, and $C$ is a smallest component of  $FQ_{n}-F$ and $ |C|=k\geq \frac{r+1}{2} +1$, we have that $|V(F)|\geq \kappa_{\frac{r+1}{2}}(FQ_{n}).$ From Ineq. (\ref{FF}) and Eq. (\ref{Fk}) we also have
$$
\aligned
\frac{1}{2}(r+1)n\geq |V(F)|\geq\frac{-r^{2}}{8}+\frac{(r+3)n}{2}
-\frac{r}{2}+\frac{5}{8},
\endaligned
$$
which implies that $n\leq\frac{r^{2}+4r-5}{8}$, contradicting $n>{\rm max}\{
6,\frac{r+5}{2},\frac{r^{2}+4r-5}{8}\}=g(r)$.
\end{proof}
\end{Lem}

From Lemma \ref{G1} we know  that the condition of Lemma \ref{FL} implies that of Lemma \ref{FU}, and thus have  the following main result of this section.
\begin{The}\label{FQ}
If $r\geq2$ and $n>g(r)$, then $\kappa(FQ_{n};K_{1,r})=\kappa^{s}(FQ_{n};K_{1,r})=\lceil\frac{n+1}{2}\rceil$
\end{The}

\section{Conclusion }
For $n$-dimensional hypercubes $Q_{n}$ and folded hypercubes $FQ_{n}$, in this paper we have showed that for all integers $r\geq 2$ and $n>f(r)$, $\kappa(Q_{n};K_{1,r})=\kappa^{s}(Q_{n}; K_{1,r})=\lceil\frac{n}{2}\rceil$, and for all integers $r\geq2$ and $n>g(r)$, $\kappa(FQ_{n};K_{1,r})=$ $\kappa^{s}(FQ_{n};$ $K_{1,r})$$=\lceil\frac{n+1}{2}\rceil$; see  Theorems \ref{hypercube} and \ref{FQ}. In particular, both functions $f(r)$ and $g(r)$ have simple expressions: $f(r)=f(r+1)=\frac{r^{2}+4r+3}{8} $ and $g(r)=g(r+1)=\frac{r^{2}+4r-5}{8}$ for odd $r\geq 9$. But for $2\leq r\leq 8$, $f(r)$ and $g(r)$ are piecewise functions with $5\leq f(r)\leq \frac{31}{3}$ and $6\leq g(r)\leq \frac{28}{3}$. Especially, for low dimensional hypercubes $Q_n$, we also obtain for all integers $4\leq r\leq6$, $\kappa(Q_{n};K_{1,r})=\kappa^{s}(Q_{n}; K_{1,r})=\lceil\frac{n}{2}\rceil$ where $n\geq r$.  For $2\leq r\leq3$, Lin et al. \cite {4} has determined  $\kappa(Q_{n},K_{1,r})$ and $\kappa^{s}(Q_{n},K_{1,r})$.
Our results solved partly  the open problem of determining $K_{1,r}$-structure connectivity of  $Q_n$ and $FQ_n$ for general $r$. Setting $r=2,3$ in Theorem \ref{FQ}, we obtain the results given by Sabir et al. in {\rm\cite {9}}.
But for the cases that $ 7\leq r\leq n\leq f(r)$ and $1\leq r-1\leq n\leq g(r)(r\geq 2;n\geq 3)$, the open problem has not been solved yet.

From the above facts obtained already we can propose the following general conjectures:
\begin{Con}For any integers $n\geq r\geq 2$ and $n\geq3$, $\kappa(Q_{n};K_{1,r})=\kappa^{s}(Q_{n}; K_{1,r})=\lceil\frac{n}{2}\rceil$.
\end{Con}
\begin{Con} For any integers $n+1\geq r\geq 2$ and $n\geq 3$,
 $\kappa(FQ_{n};K_{1,r})=\kappa^{s}(FQ_{n}; K_{1,r})=\lceil\frac{n+1}{2}\rceil$.
\end{Con}
\section*{Acknowledgement}
The author thanks the referees for many helpful suggestions and comments. This work is supported by NSFC (Grant No. 11871256).


\begin{thebibliography}{99}
    \bibitem{1}A. El-Amawy, S. Latifi, Properties and performance of folded hypercubes, IEEE Trans. Parallel Distrib. Syst.
        2 (1) (1991) 31-42.
    \bibitem{2}T. Esmaeili, G. Lak, A.N. Rad, 3D-FolH-NOC: a new structure for parallel processing and distribed systems, J. Computers 4 (6) (2012) 163-168.
    \bibitem{3}J. Fabrega, M.A. Fiod, On the extra connectivity of graphs, Discrete Math. 155 (1-3) (1996) 49-57.
    \bibitem{5}S. Latifi, Simulation of PM21 network by folded hypercube, IEE Proc. E-Comput. Digit. Tech. 138 (6) (1991) 397-400.
    \bibitem{6}D. Li, X.L. Hu, H.Q. Liu, Structure connectivity and substructure connectivity of twisted hypercubes, Theoret. Comput. Sci. 796 (2019) 169-179.
    \bibitem{4}C.K. Lin, L.L. Zhang, J.X. Fan, D.J. Wang, Structure connectivity and substructure connectivity of hypercubes, Theoret. Comput. Sci. 634 (2016) 97-107.
    \bibitem{7}Y.L. Lv, J.X. Fan, D.F. Hsu, C.K. Lin, Structure connectivity and substructure connectivity of $k$-ary $n$-cube networks, Inf. Sci. 433 (434) (2018) 115-124.
   \bibitem{18}S.A. Mane, Structure connectivity of hypercubes, AKCE Int, J. Graphs Comb. 15 (2018) 49-52.
   \bibitem{17}L. Miao, S.R. Zhang, R.H. Li, W.H. Yang, Structure fault tolerance of $k$-ary $n$-cube networks, Theoret. Comput. Sci. 795 (2019) 213-218.

   \bibitem{8}J.S. Park, N.J. Davis, The folded hypercube ATM switches, in: Proc. IEEE Int'l Conf. Networking, 2001, pp. 370-379.
   \bibitem{9}E. Sabir, J.X. Meng, Structure fault tolerance of hypercubes and folded hypercubes, Theoret. Comput. Sci. 711 (2018) 44-55.
  \bibitem{11}J.M. Xu, M.J. Ma, Cycles in folded hypercubes, Appl. Math. Lett. 19 (2) (2006) 140-145.
  \bibitem{10}J.M. Xu, Q. Zhu, X.M. Hou, T. Zhou, On restricted connectivity and extra connectivity of hypercubes and folded hypercubes, J. Shanghai Jiaotong Univ. 10 (2) (2005) 203-207.

   \bibitem{12}W.H. Yang, J.X. Meng, Extraconnectivity of hypercubes, Appl. Math. Lett. 22 (6) (2009) 887-891.
   \bibitem{13}G.Z. Zhang, D.J. Wang, Structure connectivity and substructure connectivity of bubble-sort star graph networks, Appl. Math. Comput. 363 (2019) 124632.
    \bibitem{14}M.M. Zhang, J.X. Zhou, On g-extra connectivity of folded hypercubes, Theoret. Comput. Sci. 593 (2015) 146-153.
   \bibitem{15}S.L. Zhao, W.H. Yang, Reliability evaluation of folded hypercubes in tems of component connectivity, arXiv preprint (2018) arXiv: 1803. 01311.
   \bibitem{16}Q. Zhu, J.M. Xu, X.M. Hou, M. Xu, On reliability of the folded hypercubes, Inf. Sci. 177 (8) (2007) 1782-1788.


\end{thebibliography}
\end{document}